\documentclass[a4paper]{amsart}
\usepackage{amsmath,amsthm,amssymb,latexsym,epic,bbm,comment, mathrsfs}
\usepackage{graphicx,enumerate,stmaryrd,color}
\usepackage[all,2cell]{xy}
\xyoption{2cell}

\newtheorem{theorem}{Theorem}
\newtheorem{lemma}[theorem]{Lemma}
\newtheorem{corollary}[theorem]{Corollary}
\newtheorem{proposition}[theorem]{Proposition}

\usepackage[all]{xy}
\usepackage[active]{srcltx}
\usepackage[parfill]{parskip}
\usepackage{enumerate}

\font\sc=rsfs10
\newcommand{\cC}{\sc\mbox{C}\hspace{1.0pt}}

\begin{document}
\title[Bimodules over $A_n$ quivers]
{Bimodules over uniformly oriented  $A_n$ quivers with radical square zero}

\author{Volodymyr Mazorchuk and Xiaoting Zhang}

\begin{abstract}
We start with observing that the only connected finite dimensional algebras
with finitely many isomorphism classes of indecomposable bimodules
are the quotients of the path algebras of uniformly oriented $A_n$-quivers
modulo the radical square zero relations. For such algebras we study
the (finitary) tensor category of bimodules. We describe the cell
structure of this tensor category, determine existing adjunctions
between its $1$-morphisms and find a minimal generating set with respect to
the tensor structure. We also prove that, for the algebras mentioned above,
every simple transitive $2$-representation of the $2$-category of projective
bimodules is equivalent to a cell $2$-representation.
\end{abstract}

\maketitle

\section{Introduction and description of the results}\label{s1}

Finitary $2$-categories were introduced in \cite{MM1} as ``finite dimensional''
counterparts of, in general, ``infinite dimensional'' $2$-categories which
were studied in the categorification literature on the borderline between
algebra and topology during the last twenty years, see \cite{BFK,CR,Kh,KL,Ro,St}.
The series \cite{MM1,MM2,MM3,MM4,MM5,MM6} of papers develops basics of
abstract representation theory for finitary $2$-categories. Classical
examples of finitary $2$-categories are: the $2$-category of Soergel bimodules
over the coinvariant algebra of a finite Weyl group, see \cite[Subsection~7.1]{MM1},
and the $2$-category of projective functors over finite dimensional associative
algebra, see \cite[Subsection~7.3]{MM1}. Further examples of finitary $2$-categories
were constructed and studied in \cite{GM1,GM2,Xa,Zh1,Zh2}, see also the above mentioned series
\cite{MM1}--\cite{MM6}.

In the present paper we consider a new natural class of examples of finitary $2$-categories.
The paper started from the following question:

{\em For which finite-dimensional algebras, the corresponding tensor category of bimodules is finitary?}

Although we suspect that the answer to this question could be known to specialists in
representation type, we did not manage to find any explicit answer in the literature.
In Theorem~\ref{thm1} we show that the only finite dimensional algebras over an algebraically
closed field, for which the number of isomorphism classes of bimodules is finite, are those
algebras whose connected components are radical square zero quotients of uniformly
oriented path algebras of type $A$ Dynkin quivers. This result motivates the rest of the
paper where we take a closer look at the tensor category of bimodules over such algebras.

We start with an attempt to understand the combinatorics of left-, right- and two-sided cells
of this tensor category. These cells are natural generalizations of Green's relations for semigroups
from \cite{Gr} to the setup of tensor categories, see \cite[Section~3]{MM2}. Apart from
$\Bbbk$-split bimodules, that is bimodules of the form $X\otimes_{\Bbbk}Y$, where
$X$ is a left module and $Y$ is a right module, see \cite{MMZ}, the remaining bimodules
can be split into four families, which we call $W$, $M$, $N$ and $S$, motivated by the
shape of the diagram of a bimodule. All $\Bbbk$-split bimodules always form the maximum
two-sided cell which is easy to understand.

To describe the remaining structure, we introduce several combinatorial invariants of bimodules,
called {\em left support} and {\em right support} and also the number of {\em valleys}  in
the diagram of a bimodule. We show that these invariants, in combination with bimodule types,
classify left, right and two-sided cells. For example, two-sided cells are classified,
in the case of non $\Bbbk$-split bimodules, by the number of valleys in bimodule diagrams.
This result is a first step in understanding combinatorial structure for bimodule categories
over arbitrary finite dimensional algebras, the latter question forming the core of our motivation.
We also give an explicit description for all adjoint pairs of functors formed by our bimodules.
This description covers only some bimodules as, in general, the right adjoint of
tensoring with a bimodule is not exact and hence is not isomorphic to tensoring with some
(possibly different) bimodule.

Furthermore, we find a minimal generating set for our tensor category, with respect to the
tensor structure. It consists of the identity bimodule and three additional bimodules in
the two-sided cell closest to the one formed by the identity bimodule, with respect to the
two-sided order. To prove the statement, we give explicit formulae for tensor products 
of each of these three bimodules  with all other indecomposable bimodules.

Finally, we study simple transitive   of the $2$-category of projective
bimodules over our algebras. Classification of such $2$-representations is a natural 
problem which was considered, for
various classes of $2$-categories in \cite{MM5,MM6,Zh1,Zi,MZ,MaMa,KMMZ,MT,MMMT,MMZ}.
It also has interesting applications, see \cite{KM1}. A natural class of simple transitive
$2$-representations is given by the so-called {\em cell $2$-representations} constructed in \cite{MM1,MM2}.
For the $2$-category of projective bimodules over a finite dimensional algebra $A$, it is known that
cell $2$-representations exhaust all simple transitive $2$-representations if $A$ is self-injective
(see \cite{MM5,MM6}), if $A=\Bbbk[x,y]/(x^2,y^2,xy)$ (see \cite{MMZ}) and if
$A$ is the radical square zero quotient of the path algebra of a uniformly oriented quiver of type
$A_2$ or $A_3$ (see \cite{MZ}). In the present paper we extend this result to all directed
algebras admitting a non-zero projective-injective module, see Theorem~\ref{thmmain3}.
We recover, with a much shorter and much more elegant proof, the main results of \cite{MZ},
however, our approach is strongly inspired by \cite[Section~3]{MZ}. Our approach to the proof of
Theorem~\ref{thmmain3} contains some new general ideas which could help to attack similar problems
for other finitary $2$-categories. We note that there are natural examples of $2$-categories
which have simple transitive $2$-representations that are not equivalent to cell $2$-representations,
see \cite{MaMa,KMMZ,MT,MMMT}.

The paper is organized as follows: Section~\ref{s2} contains the material related
to the formulation and proof of Theorem~\ref{thm1}. Section~\ref{s3} studies combinatorics
of bimodules. The main results of this section which provide a combinatorial description
of the cell structure are collected in Subsection~\ref{s3.27}. In Theorem~\ref{thmw-1} of
Section~\ref{snew}, we describe a minimal generating set of our tensor category with
respect to the tensor structure. Finally, Section~\ref{s4}
contains the material related to classification of simple transitive $2$-representations.
\vspace{5mm}

\textbf{Acknowledgements:} This research was partially supported by
the Swedish Research Council, Knut and Alice Wallenberg Stiftelse and
G{\"o}ran Gustafsson Stiftelse.

\section{Characterization via representation type}\label{s2}

\subsection{Main object of study}\label{s2.1}

Throughout the paper we fix an algebraically closed field $\Bbbk$. For $n\in\{1,2,3,\dots\}$,
we denote by $A_n$ the quotient of the path algebra of the quiver
\begin{equation}\label{eq1}
\xymatrix{
\mathtt{1}\ar[r]^{\alpha_1}&\mathtt{2}\ar[r]^{\alpha_2}&
\mathtt{3}\ar[r]^{\alpha_3}&\dots\ar[r]^{\alpha_{n-1}}&\mathtt{n}
}
\end{equation}
modulo the relations that the product of any two arrows is zero.  In particular, we have
$A_1\cong \Bbbk$, $A_2$ is isomorphic to the algebra of upper triangular $2\times 2$ matrices
with coefficients in $\Bbbk$ and $\mathrm{Rad}(A_n)^2=0$, for any $n$. If $n$ is fixed or
clear from the context, we will simply write $A$ for $A_n$.

We denote by
\begin{itemize}
\item $A$-mod the category of finite dimensional left $A$-modules;
\item mod-$A$ the category of finite dimensional right $A$-modules;
\item $A$-mod-$A$ the category of finite dimensional $A$-$A$-bimodules.
\end{itemize}

For $i=1,2,\dots,n$, we denote by $e_i$ the trivial path at the vertex $\mathtt{i}$.
Thus we have a primitive decomposition $1=e_1+e_2+\dots+e_n$ of the identity $1\in A$.
Then $P_i=Ae_i$ is an indecomposable projective in $A$-mod and we denote by $L_i$
the simple top of $P_i$. Further, we denote by $I_i$ the indecomposable injective
envelope of $L_i$. Note that $P_i$ has dimension $2$, for $i=1,2,\dots,n-1$,
and $P_n=L_n$. Similarly, $I_i$ has dimension $2$, for $i=2,3,\dots,n$,
and $I_1=L_1$. Moreover, $P_i\cong I_{i+1}$, for $i=1,2,\dots,n-1$.

\subsection{Bimodule representation type}\label{s2.2}

The main result of this subsection is the following statement. We suspect that this claim
should be known to experts, but we failed to find any explicit reference in the literature.

\begin{theorem}\label{thm1}
Let $B$ be a finite dimensional associative $\Bbbk$-algebra. Then the following conditions
are equivalent:
\begin{enumerate}[$($a$)$]
\item\label{thm1.1} The category $B$-mod-$B$ has finitely many isomorphism classes of
indecomposable objects.
\item\label{thm1.2} Each connected component of $B$ is Morita equivalent to some $A_n$.
\end{enumerate}
\end{theorem}

As $B$-mod-$B$ is equivalent to $B\otimes_{\Bbbk} B^{\mathrm{op}}$-mod, condition \eqref{thm1.1}
is equivalent to saying that $B\otimes_{\Bbbk} B^{\mathrm{op}}$ is of finite representation type.

\subsection{Proof of the implication \eqref{thm1.1}$\Rightarrow$\eqref{thm1.2}}\label{s2.3}

Let $B$ be a basic finite dimensional associative $\Bbbk$-algebra such that the category $B$-mod-$B$ has
finitely many isomorphism classes of indecomposable objects. Note that this, in particular, implies
that $B$ has finite representation type. Consider the Gabriel quiver $Q=(Q_0,Q_1)$ of $B$,
where $Q_0$ is the set of vertices and $Q_1$ the set of arrows. Then, for any $i,j\in Q_0$, we have at
most one arrow from $i$ to $j$ for otherwise $B$ would surject onto the
Kronecker algebra and thus have infinite representation type.

Next we claim that $Q$ has no loops. Indeed, if $Q$ would have a loop, $B$ would have a quotient
isomorphic to the algebra $D:=\Bbbk(x)/(x^2)$ of dual numbers. However, the algebra
$D\otimes_{\Bbbk}D^{\mathrm{op}}$ has a quotient isomorphic to $\Bbbk(x,y)/(x^2,y^2,xy)$ and the
latter has infinite representation type since we have an infinite family of pairwise non-isomorphic
indecomposable $2$-dimensional modules of this algebra on which $x$ and $y$ act via
\begin{displaymath}
\left(\begin{array}{cc}0&1\\0&0\end{array}\right)\quad\text{ and }\quad
\left(\begin{array}{cc}0&\lambda\\0&0\end{array}\right), \quad\text{ where }\lambda\in\Bbbk,
\end{displaymath}
respectively (note that $\Bbbk$, being algebraically closed, is infinite). This is a contradiction.

Next we claim that each vertex of $Q$ has indegree at most $1$ and outdegree at most $1$.
We prove the claim for indegrees and, for outdegrees, the proof is similar. If $Q$ has a vertex with
indegree at least two, then, taking the above into account, $Q$ has a subgraph of the form:
\begin{equation}\label{eq2}
\xymatrix{
i\ar[r]&j&k\ar[l]
}
\end{equation}
Then $B$ has a quotient isomorphic to the path algebra $C$ of \eqref{eq2}. We claim that
$C\otimes_{\Bbbk}C^{\mathrm{op}}$ has infinite representation type thus giving us a contradiction.
Indeed, $C\otimes_{\Bbbk}C^{\mathrm{op}}$ is isomorphic to the quotient of the path algebra of
the solid part of the following quiver:
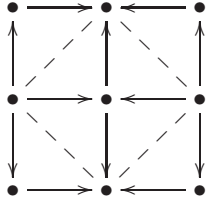
\begin{displaymath}
\xymatrix{
\bullet\ar[r]&\bullet&\bullet\ar[l]\\
\bullet\ar[r]\ar[u]\ar[d]\ar@{--}[dr]\ar@{--}[ur]&\bullet\ar[u]\ar[d]&
\bullet\ar[l]\ar[u]\ar[d]\ar@{--}[dl]\ar@{--}[ul]\\
\bullet\ar[r]&\bullet&\bullet\ar[l]
}
\end{displaymath}
modulo the commutativity relations indicated by dotted lines. In the middle of this picture we
see an orientation of the affine Dynkin diagram of type $\tilde{D}_4$. The corresponding
centralizer subalgebra thus has infinite representation type. It follows that
$C\otimes_{\Bbbk}C^{\mathrm{op}}$ has infinite representation type.

Next we claim that $Q$ does not have any components of the form
$\xymatrix{\bullet\ar@/^/[r]&\bullet\ar@/^/[l]}$. If the latter were the case, then
$B\otimes_{\Bbbk}B^{\mathrm{op}}$ would have a quotient isomorphic to the path algebra of
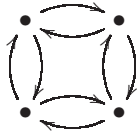
\begin{displaymath}
\xymatrix{
\bullet\ar@/^/[r]\ar@/^/[d]&\bullet\ar@/^/[l]\ar@/^/[d]\\
\bullet\ar@/^/[r]\ar@/^/[u]&\bullet\ar@/^/[l]\ar@/^/[u]
}
\end{displaymath}
modulo the relations that all squares commute. The latter has a quotient isomorphic to the
path algebra corresponding to the following orientation
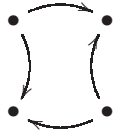
\begin{displaymath}
\xymatrix{
\bullet\ar@/^/[r]\ar@/^/[d]&\bullet\\
\bullet&\bullet\ar@/^/[l]\ar@/^/[u]
}
\end{displaymath}
of an affine Dynkin quiver of type $\tilde{A}_4$ and thus has infinite representation type,
a contradiction.

The above shows that $Q$ is a disjoint union of graphs of the form \eqref{eq1}. We now only need
to show that $\mathrm{Rad}(B)^2=0$. If $\mathrm{Rad}(B)^2\neq 0$, then
$B$ has a quotient which is isomorphic to the path
algebra $F$ of \eqref{eq1}, for $n=3$. Then $F\otimes_{\Bbbk}F^{\mathrm{op}}$ is isomorphic to
the quotient of the path algebra of the solid part of the following quiver:
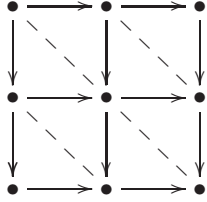
\begin{displaymath}
\xymatrix{
\bullet\ar[r]\ar[d]\ar@{--}[dr]&\bullet\ar[r]\ar[d]\ar@{--}[dr]&\bullet\ar[d]\\
\bullet\ar[r]\ar[d]\ar[d]\ar@{--}[dr]&\bullet\ar[r]\ar[d]\ar@{--}[dr]&
\bullet\ar[d]\\
\bullet\ar[r]&\bullet\ar[r]&\bullet
}
\end{displaymath}
modulo the commutativity relations indicated by dotted lines.
Similarly to the previous paragraph, this algebra has a centralizer subalgebra which is the path
algebra of a type $\tilde{D}_4$ quiver. Therefore it has infinite representation type.

\subsection{Proof of the implication \eqref{thm1.2}$\Rightarrow$\eqref{thm1.1}}\label{s2.4}

Note that $A_n/(e_n)\cong A_{n-1}$, for any $n$. Therefore, for all $m,n$, there is a full
embedding of the category of $A_n$-$A_m$-bimodules into the category of $A_k$-$A_k$-bimodules,
where $k=\max(m,n)$. Using additivity and the fact that Morita equivalence, by definition,
does not affect representation type, we obtain that it is sufficient to prove that
the category $A_n$-mod-$A_n$ has finitely many isomorphism classes of indecomposable objects,
for every $n$.

The algebra $A_n\otimes_{\Bbbk} A_n^{\mathrm{op}}$ is the quotient of the path algebra of the
following quiver:
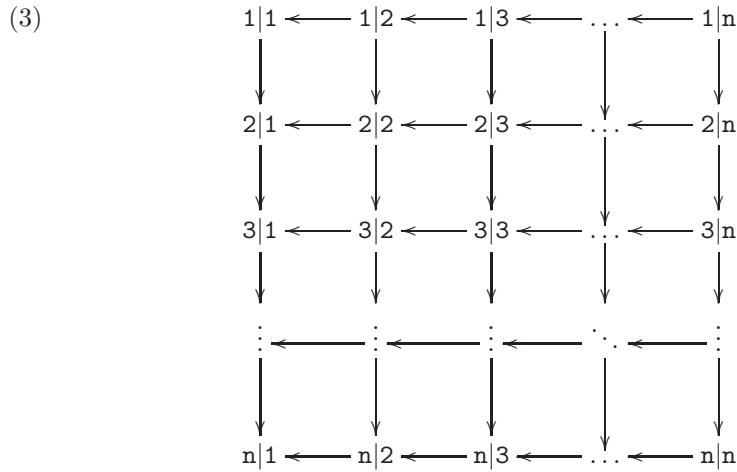
\begin{equation}\label{eq0}
\xymatrix{
\mathtt{1\vert1}\ar[d]&\mathtt{1\vert2}\ar[l]\ar[d]&
\mathtt{1\vert3}\ar[l]\ar[d]&\dots\ar[l]\ar[d]&\mathtt{1\vert n}\ar[d]\ar[l]\\
\mathtt{2\vert1}\ar[d]&\mathtt{2\vert2}\ar[l]\ar[d]&
\mathtt{2\vert3}\ar[l]\ar[d]&\dots\ar[l]\ar[d]&\mathtt{2\vert n}\ar[d]\ar[l]\\
\mathtt{3\vert1}\ar[d]&\mathtt{3\vert2}\ar[l]\ar[d]&
\mathtt{3\vert3}\ar[l]\ar[d]&\dots\ar[l]\ar[d]&\mathtt{3\vert n}\ar[d]\ar[l]\\
\vdots\ar[d]&\vdots\ar[l]\ar[d]&
\vdots\ar[l]\ar[d]&\ddots\ar[l]\ar[d]&\vdots\ar[d]\ar[l]\\
\mathtt{n\vert1}&\mathtt{n\vert2}\ar[l]&
\mathtt{n\vert3}\ar[l]&\dots\ar[l]&\mathtt{n\vert n}\ar[l]
}
\end{equation}
modulo the relations that all squares commute and the product of any two horizontal or any two
vertical arrows is zero.

The algebra $A_n\otimes_{\Bbbk} A_n^{\mathrm{op}}$ is thus a special biserial algebra in the sense of \cite{BR,WW}.
According to these two references, each indecomposable module over a special biserial algebra is
either a string modules or a band module or a non-uniserial projective-injective module. In our case,
there are certainly finitely many indecomposable non-uniserial projective-injective modules
(they correspond to commutative squares in our quiver).

We claim that, in the case of $A_n\otimes_{\Bbbk} A_n^{\mathrm{op}}$, there are only finitely many string modules (up to isomorphism) and that there are no band modules.
This follows from the form of the relations.
Indeed, the maximal (with respect to inclusion) strings avoiding zero relations are:
\begin{equation}\label{eq3}
\xymatrix@C=1.6em@R=1.6em{
\mathtt{1\vert n\text{-}1}&\ar[d]\mathtt{1\vert n}\ar[l]\\
&\mathtt{2\vert n}
}\qquad
\xymatrix@C=1.6em@R=1.6em{
\mathtt{1\vert n\text{-}2}&\mathtt{1\vert n\text{-}1}\ar[l]\ar[d]\\
&\mathtt{2\vert n\text{-}1}&\mathtt{2\vert n}\ar[l]\ar[d]\\
&&\mathtt{3\vert n}
}\qquad
\xymatrix@C=1.6em@R=1.6em{
\mathtt{1\vert n\text{-}3}&\mathtt{1\vert n\text{-}2}\ar[l]\ar[d]\\
&\mathtt{2\vert n\text{-}2}&\mathtt{2\vert n\text{-}1}\ar[l]\ar[d]\\
&&\mathtt{3\vert n\text{-}1}&\mathtt{3\vert n}\ar[l]\ar[d]\\
&&&\mathtt{4\vert n}
}
\end{equation}
and so on (in total, there are $2n-2$ such maximal strings). We see that edges of these strings never
intersect
and that the strings never close into bands (corresponding to primitive cyclic words in some references).
Consequently, there are
no band modules and
only finitely many string modules.
The claim follows.

An exact enumeration of isomorphism classes of indecomposable objects in the category
$A_n$-mod-$A_n$ is given in the next subsection.

\subsection{Enumeration of indecomposable $A_n$-$A_n$-bimodules}\label{s2.5}

\begin{proposition}\label{prop2}
For $n=1,2,\dots$, the category $A_n$-mod-$A_n$ contains exactly
\begin{displaymath}
\frac{4n^3+3n^2-7n+3}{3}
\end{displaymath}
isomorphism classes of indecomposable objects.
\end{proposition}

\begin{proof}
We have $(n-1)^2$ indecomposable projective-injective objects in $A_n$-mod-$A_n$. From Subsection~\ref{s2.4},
we know that the remaining indecomposable objects correspond to string $A_n\otimes_{\Bbbk} A_n^{\mathrm{op}}$-module.
A string module is uniquely identified by the string the module is supported at, that is by a substring of
one of the maximal string as shown in \eqref{eq3}.

There are two maximal strings with three vertices, two maximal strings with $5$ vertices and so on.
The maximal number of vertices on a maximal string is $2n-1$. A string with $k$ vertices supports
$k(k+1)/2$ string modules (which correspond to connected substrings). Note that simple modules are
supported just by vertices and there are $n^2-2$ vertices, that is all but the left upper and right lower
corners, which belong to two different maximal substrings and hence are counted twice above.
Putting all this together and simplifying gives the necessary formula.
\end{proof}

\section{Combinatorics of $A$-$A$-bimodules}\label{s3}

\subsection{Cells}\label{s3.0}

The main aim of the section is to describe the cell combinatorics of $A$-$A$-bimodules
in the sense of \cite[Section~3]{MM2}. Denote by $\mathcal{S}$ the set of isomorphism classes
of indecomposable $A$-$A$-bimodules. Recall that the {\em left preorder} $\leq_L$ is defined
as follows: for $X,Y\in \mathcal{S}$ we have $X\leq_L Y$ provided that $Y$ is isomorphic to
a direct summand of $Z\otimes_A X$, for some $A$-$A$-bimodule $X$. An equivalence class
with respect to $\leq_L$ is called a {\em left cell} and the corresponding equivalence relation
is denoted by $\sim_L$. The {\em right preorder} $\leq_R$ and the corresponding
{\em right cells} and $\sim_R$ are defined similarly using tensoring over $A$ from the right.
The {\em two-sided preorder} $\leq_J$ and the corresponding
{\em two-sided cells} and $\sim_J$ are defined similarly using tensoring over $A$ from both sides.

\subsection{$\Bbbk$-split $A$-$A$-bimodules}\label{s3.1}

For simplicity, in this section we will denote by $\mathbb{N}_n$ the set
$\{0,1,2,\dots,n-1\}$ and $\mathbb{N}_n^\ast$
the set $\{1,2,\dots,n-1\}$.

An $A$-$A$-bimodule $X$ is called {\em $\Bbbk$-split}, cf. \cite{MMZ},
provided that $X$ is isomorphic to a direct sum of $A$-$A$-bimodules of the form
$M\otimes_{\Bbbk}N$, where $M\in A$-mod and $N\in\mathrm{mod}\text{-}A$.

We will often argue using bimodule action graphs. We will depict the left
action using vertical arrows and the right action using horizontal arrows.
Following the proof of Proposition~\ref{prop2}, we can now list
all non-zero indecomposable $\Bbbk$-split $A$-$A$-bimodules (up to isomorphism)
and describe their action graphs as follows:
\begin{itemize}
\item the projective-injective bimodules $Ae_i\otimes_{\Bbbk}e_{j+1}A$, where  $i,j\in\mathbb{N}_n^\ast$:
\begin{displaymath}
\xymatrix@C=1.6em@R=1.6em{
\mathtt{i\mid j}\ar[d]&\mathtt{i\mid j\text{+}1}\ar[l]\ar[d]\\
\mathtt{i\text{+}1\mid j}&*+[r]{\mathtt{i\text{+}1\mid j\text{+}1};}\ar[l]
}
\end{displaymath}
\item the simple bimodules $\mathtt{i\mid j}$, where $ i,j\in\mathbb{N}_{n+1}^\ast$, that is, string $A\otimes_{\Bbbk} A^{\mathrm{op}}$-modules of dimension $1$ (we will denote the bimodule $\mathtt{i\mid j}$ by $L_{ij}$);
\item string $A\otimes_{\Bbbk} A^{\mathrm{op}}$-modules of dimension $2$:
\begin{displaymath}
\xymatrix@C=1.6em@R=1.6em{
\mathtt{i\mid j,}\ar[d],\\
\mathtt{i\text{+}1\mid j}
}\quad i\in\mathbb{N}_{n}^\ast,j\in \mathbb{N}_{n+1}^\ast;\qquad
\xymatrix@C=1.6em@R=1.6em{
\mathtt{i\mid j}&\mathtt{i\mid j\text{+}1,}\ar[l]
}\quad i\in\mathbb{N}_{n+1}^\ast,j\in \mathbb{N}_{n}^\ast.
\end{displaymath}
\end{itemize}
Note that the projective $A$-$A$-bimodules $Ae_n\otimes_{\Bbbk}e_{j+1}A$, where $j\in \mathbb{N}_{n}^\ast$,
and also $Ae_i\otimes_{\Bbbk}e_1A$, where $i\in \mathbb{N}_{n}^\ast$, belong to the last type of indecomposables
(and they are not injective).

We denote by $\mathcal{J}_{\Bbbk}$ the set of all $\Bbbk$-split elements in $\mathcal{S}$.

\begin{proposition}\label{propn-2}
{\small\hspace{1mm}}

\begin{enumerate}[$($i$)$]
\item\label{propn-2.1} The set  $\mathcal{J}_{\Bbbk}$ is the unique maximal
two-sided cell in $\mathcal{S}$.
\item\label{propn-2.2} For each indecomposable right $A$-module $N$, the set of elements in
$\mathcal{J}_{\Bbbk}$ of the form $X\otimes_{\Bbbk} N$, for some $X\in A$-mod, forms a left cell.
Moreover, each left cell in $\mathcal{J}_{\Bbbk}$ is of such form.
\item\label{propn-2.3} For each indecomposable left $A$-module $K$, the set of elements in
$\mathcal{J}_{\Bbbk}$ of the form $K\otimes_{\Bbbk} Y$, for some $Y\in \mathrm{mod}\text{-}A$, forms a right cell.
Moreover, each right cell in $\mathcal{J}_{\Bbbk}$ is of such form.
\end{enumerate}
\end{proposition}

\begin{proof}
For $X\in \mathcal{J}_{\Bbbk}$ and $Y$ any $A$-$A$-bimodule, both $X\otimes_A Y$ and $Y\otimes_A X$
are obviously $\Bbbk$-split, so $\Bbbk$-split bimodules form a tensor ideal in $A$-mod-$A$.
For any indecomposable $K_1$ and $K_2$ in $A$-mod and any indecomposable $X\in \mathrm{mod}\text{-}A$,
we have
\begin{displaymath}
\left(K_1\otimes_{\Bbbk}A\right)\otimes_A \left(K_2\otimes_{\Bbbk} X\right)\cong
\left(K_1\otimes_{\Bbbk} X\right)^{\oplus\dim K_2}.
\end{displaymath}
This implies claim~\eqref{propn-2.2} and claim~\eqref{propn-2.3} is proved similarly.
This, combined with the fact that $\Bbbk$-split bimodules form a tensor ideal in $A$-mod-$A$,
also implies claim~\eqref{propn-2.1}, completing the proof.
\end{proof}

\subsection{$A$-$A$-bimodules which are not $\Bbbk$-split }\label{s3.2}

Consider the action graph $\Gamma_M$ for a string $A\otimes_{\Bbbk} A^{\mathrm{op}}$-module $M$.
Then, directly from the construction of string modules, we can make the following easy observations:
\begin{itemize}
\item the indegree of each vertex in $\Gamma_M$ is at most two;
\item the outdegree of each vertex in $\Gamma_M$ is at most two;
\item each vertex of $\Gamma_M$ is either a source or a sink (or both);
\item $M$ is simple if and only if both the indegree and the outdegree of each vertex in $\Gamma_M$ is zero.
\end{itemize}
A vertex of $\Gamma_M$ of indegree exactly two will be called  a {\em valley}.
A typical example of a valley in an action graph is the white vertex of the following graph:
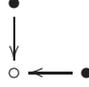
\begin{displaymath}
\xymatrix@C=1.6em@R=1.6em{
\bullet\ar[d]&\\
\circ&\bullet\ar[l]
}
\end{displaymath}
We denote by $\mathbf{v}(M)$ the number of valleys in $\Gamma_M$. Clearly, $0\leq \mathbf{v}(M)\leq n-1$,
moreover, the only $M$ for which $\mathbf{v}(M)= n-1$ is the regular bimodule $M\cong {}_AA_A$.
For any  $A\otimes_{\Bbbk} A^{\mathrm{op}}$-module $N$ which is not projective-injective, define
\begin{displaymath}
\mathbf{v}(N):=\text{max}\{\mathbf{v}(M):M \text{ is a string module and is as a direct summand of }N\}.
\end{displaymath}

Indecomposable $A$-$A$-bimodules that are not $\Bbbk$-split
correspond to string $A\otimes_{\Bbbk} A^{\mathrm{op}}$-modules whose action
graphs have $k$ vertices, where $3\leq k\leq 2n-1$. Below we list all such bimodules,
fix notation for them and describe the corresponding action graphs. We use
the number of valleys in the action graph as a parameter, denoted by $t$.
Our choice for notation is motivated by the shape of the action graph.

{\bf Bimodules $W_{ij}^{t}$.}
For any $t\in\mathbb{N}_{n}^\ast$ and $i,j\in\mathbb{N}_{n-t+1}^\ast$, we
denote by $W_{ij}^{t}$ the following $A$-$A$-bimodule:
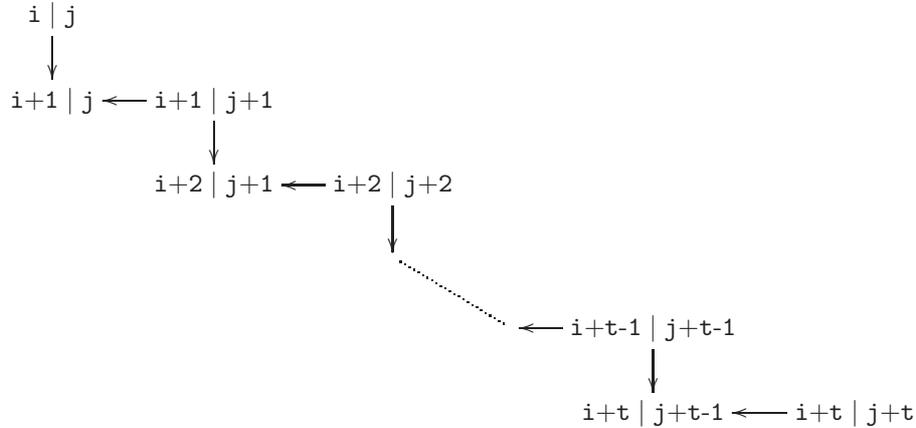
\begin{displaymath}
\xymatrix@C=1.6em@R=1.6em{
\mathtt{i\mid j}\ar[d]&&&&&\\
\mathtt{i\text{+}1\mid j}&\mathtt{i\text{+}1\mid j\text{+}1}\ar[l]\ar[d]&&&&\\
&\mathtt{i\text{+}2\mid j\text{+}1}&\mathtt{i\text{+}2\mid j\text{+}2}\ar[l]\ar[d]&&&\\
&&\ar@{.}[rd]&&&\\
&&&&\mathtt{i\text{+}t\text{-}1\mid j\text{+}t\text{-}1}\ar[l]\ar[d]&\\
&&&&\mathtt{i\text{+}t\mid j\text{+}t\text{-}1}&\mathtt{i\text{+}t\mid j\text{+}t}\ar[l]}
\end{displaymath}
In particular, we have ${}_AA_A\cong W_{11}^{n-1}$.

{\bf Bimodules $S_{ij}^{t}$.}
For any $t\in\mathbb{N}_{n-1}^\ast$, $i\in\mathbb{N}_{n-t}^\ast $, and $j\in\mathbb{N}_{n-t+1}^\ast$,
we  denote by $S_{ij}^{t}$ the following $A$-$A$-bimodule:
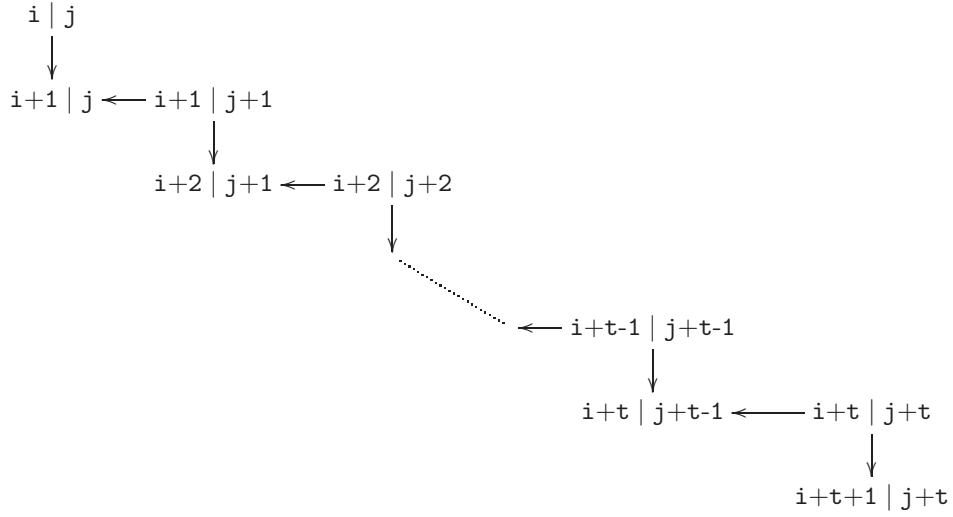
\begin{displaymath}
\xymatrix@C=1.6em@R=1.6em{
\mathtt{i\mid j}\ar[d]&&&&&\\
\mathtt{i\text{+}1\mid j}&\mathtt{i\text{+}1\mid j\text{+}1}\ar[l]\ar[d]&&&&\\
&\mathtt{i\text{+}2\mid j\text{+}1}&\mathtt{i\text{+}2\mid j\text{+}2}\ar[l]\ar[d]&&&\\
&&\ar@{.}[dr]&&&\\
&&&&\mathtt{i\text{+}t\text{-}1\mid j\text{+}t\text{-}1}\ar[l]\ar[d]&\\
&&&&\mathtt{i\text{+}t\mid j\text{+}t\text{-}1}&\mathtt{i\text{+}t\mid j\text{+}t}\ar[l]\ar[d]\\
&&&&&\mathtt{i\text{+}t\text{+}1\mid j\text{+}t}}
\end{displaymath}

{\bf Bimodules $N_{ij}^{t}$.}
For any $t\in\mathbb{N}_{n-1}^\ast$, $i\in\mathbb{N}_{n-t+1}^\ast$ and $j\in\mathbb{N}_{n-t}^\ast$,
we denote by $N_{ij}^{t}$ the following $A$-$A$-bimodule:
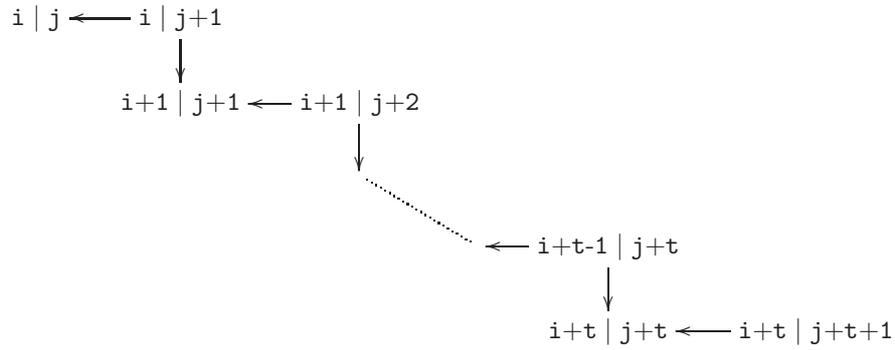
\begin{displaymath}
\xymatrix@C=1.6em@R=1.6em{
\mathtt{i\mid j}&\mathtt{i\mid j\text{+}1}\ar[l]\ar[d]&&&&\\
&\mathtt{i\text{+}1\mid j\text{+}1}&\mathtt{i\text{+}1\mid j\text{+}2}\ar[l]\ar[d]&&&\\
&&\ar@{.}[dr]&&&\\
&&&&\mathtt{i\text{+}t\text{-}1\mid j\text{+}t}\ar[l]\ar[d]&\\
&&&&\mathtt{i\text{+}t\mid j\text{+}t}&\mathtt{i\text{+}t\mid j\text{+}t\text{+}1}\ar[l]
}
\end{displaymath}

{\bf Bimodules $M_{ij}^{t}$.}
For any $t\in\mathbb{N}_{n-1}$ and $i,j\in\mathbb{N}_{n-t}^\ast$, we
denote by $M_{ij}^{t}$ the following $A$-$A$-bimodule:
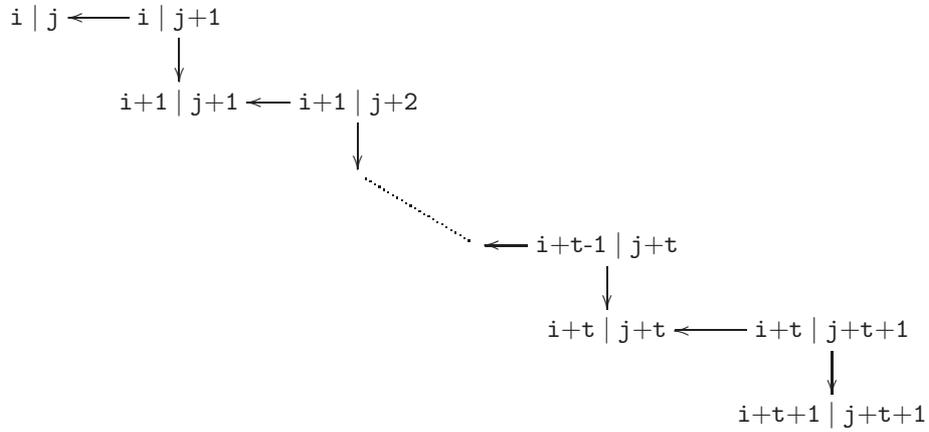
\begin{displaymath}
\xymatrix@C=1.6em@R=1.6em{
\mathtt{i\mid j}&\mathtt{i\mid j\text{+}1}\ar[l]\ar[d]&&&&\\
&\mathtt{i\text{+}1\mid j\text{+}1}&\mathtt{i\text{+}1\mid j\text{+}2}\ar[l]\ar[d]&&&\\
&&\ar@{.}[dr]&&&\\
&&&&\mathtt{i\text{+}t\text{-}1\mid j\text{+}t}\ar[l]\ar[d]&\\
&&&&\mathtt{i\text{+}t\mid j\text{+}t}&\mathtt{i\text{+}t\mid j\text{+}t\text{+}1}\ar[l]\ar[d]\\
&&&&&\mathtt{i\text{+}t\text{+}1\mid j\text{+}t\text{+}1}
}
\end{displaymath}
In particular, $\mathrm{Hom}_{\Bbbk}({}_AA_A,\Bbbk)\cong M_{11}^{n-2}$.
As a concrete example, the action graph of the bimodule $M_{ij}^{0}$, where $ i,j\in\mathbb{N}_{n}^\ast$, is:
\begin{displaymath}
\xymatrix@C=1.6em@R=1.6em{
\mathtt{i\mid j}&\mathtt{i\mid j\text{+}1}\ar[l]\ar[d]\\
&\mathtt{i\text{+}1\mid j\text{+}1}
}
\end{displaymath}

Let $M$ be an indecomposable $A$-$A$-bimodule listed in Subsection~\ref{s3.1}-\ref{s3.2}.
Then vertices in $\Gamma_M$ correspond to the {\em standard basis} of $M$ denoted $\mathbf{B}(M)$.
We will often identify vertices in $\Gamma_M$ with this standard basis.
Note the following:
\begin{itemize}
\item every vertex of non-zero indegree generates a simple sub-bimodule;
\item every vertex of outdegree two generates a subbimodule which is isomorphic to some $M_{ij}^0$.
\end{itemize}

\subsection{Supports}\label{s3.25}

For an $A$-$A$-bimodule $X$, we define the {\em left support} of $X$ as the set of all $i\in\{1,2,\dots,n\}$
such that $e_iX\neq 0$. The left support of $X$ will be denoted $\mathrm{Lsupp}(X)$.
Similarly, the {\em right support} of $X$, denoted $\mathrm{Rsupp}(X)$,  is the set of all
$i\in\{1,2,\dots,n\}$ such that $Xe_i\neq 0$. Note that, for any indecomposable
$A$-$A$-bimodule $X$, both $\mathrm{Lsupp}(X)$ and $\mathrm{Rsupp}(X)$ are convex subset
of $\{1,2,\dots,n\}$. Here $X\subset \{1,2,\dots,n\}$ is {\em convex} if,
for any $x,y,z\in \{1,2,\dots,n\}$, the properties $x,z\in X$ and $x<y<z$ imply $y\in X$.

We define the {\em width} $\mathbf{w}(X)$  and the {\em height} $\mathbf{h}(X)$ of $X$ to be, respectively,
\begin{gather*}
\mathbf{w}(X):=1+\max\{i\in \mathrm{Rsupp}(X)\}-\min\{i\in \mathrm{Rsupp}(X)\},\\
\mathbf{h}(X):=1+\max\{i\in \mathrm{Lsupp}(X)\}-\min\{i\in \mathrm{Lsupp}(X)\}.
\end{gather*}
Note that, for an indecomposable $X$, we have $|\mathbf{w}(X)-\mathbf{h}(X)|\leq 1$.

It is worth noting that indecomposable bimodules in the families $N$ and $S$ are uniquely
determined (inside the set of isomorphism classes of indecomposable $A$-$A$-bimodules) by their
left and right supports and that for bimodules in these families we always have
$|\mathbf{w}(X)-\mathbf{h}(X)|=1$. Further, each indecomposable bimodule  in the families $M$
and  $W$ is uniquely determined by its left and right support inside its family. Moreover,
for each bimodule $X$ in the families $M$ and $W$, there is a unique bimodule in the other family
(i.e. in $W$ if $X$ is in $M$ and vice versa) which shares the  left and right supports with $X$.
Note that, if $X$ of the $M$-family and $Y$ of the $W$-family  share the  left and right supports,
then $\mathbf{v}(Y)>\mathbf{v}(X)$.

The following lemma contains a crucial observation for our combinatorial description. The claim
follows directly from the definitions.

\begin{lemma}\label{lemn-1}
For any $A$-$A$-bimodules $X$ and $Y$, we have
\begin{displaymath}
\mathrm{Lsupp}(X\otimes_A Y)\subset \mathrm{Lsupp}(X)\quad\text{ and }\quad
\mathrm{Rsupp}(X\otimes_A Y)\subset \mathrm{Rsupp}(Y).
\end{displaymath}
\end{lemma}

\subsection{Description of cells}\label{s3.27}

For $v\in\mathbb{N}_n$, denote by $\mathcal{J}_v$ the set of all non $\Bbbk$-split elements
in $\mathcal{S}$ having exactly $v$ valleys. Note that $\mathcal{J}_{n-1}$ consists just of
the identity $A$-$A$-bimodule  ${}_AA_A$.
Recall from \cite[Subsection~2.3]{CM} that a two-sided cell
is called {\em idempotent} provided that it contains three elements
$\mathrm{F}$, $\mathrm{G}$ and $\mathrm{H}$ (not necessarily different), such that
$\mathrm{F}$ is isomorphic to a direct summand of $\mathrm{G}\circ\mathrm{H}$.

In this subsection we present our main results. We start by describing two-sided cells in $\mathcal{S}$.
\vspace{1cm}

\begin{theorem}\label{thmn-3}
{\hspace{2mm}}

\begin{enumerate}[$($i$)$]
\item\label{thmn-3.1} Each $\mathcal{J}_v$ is a two-sided cell in $\mathcal{S}$, moreover,
each two-sided cell in $\mathcal{S}$ coincides with either $\mathcal{J}_{\Bbbk}$ or with
one of the $\mathcal{J}_v$'s.
\item\label{thmn-3.2} The two-sided cells of  $\mathcal{S}$ are linearly ordered as follows:
\begin{displaymath}
\mathcal{J}_{\Bbbk}\geq_J\mathcal{J}_0\geq_J\mathcal{J}_1\geq_J\dots \geq_J\mathcal{J}_{n-1}.
\end{displaymath}
\item\label{thmn-3.3} All two-sided cells but $\mathcal{J}_0$ are idempotent.
\end{enumerate}
\end{theorem}

As left and right cells in $\mathcal{J}_{\Bbbk}$ are already described in Proposition~\ref{propn-2},
it is left to consider left and right cells formed by non $\Bbbk$-split bimodules.

\begin{theorem}\label{thmn-4}
Let $v\in \mathbb{N}_n$.

\begin{enumerate}[$($i$)$]
\item\label{thmn-4.1} If $v>0$, then, for each $j\in\mathbb{N}_{n-v+1}^*$, the set of all bimodules of
type $W$ and $S$ in $\mathcal{J}_v$ with right support $\{j,j+1,\dots,j+v\}$ forms a left cell in $\mathcal{J}_v$.
\item\label{thmn-4.2} For each $j\in\mathbb{N}_{n-v}^*$, the set of all bimodules of
type $M$ and $N$ in $\mathcal{J}_v$ with right support $\{j,j+1,\dots,j+v+1\}$ forms a left cell in $\mathcal{J}_v$.
\item\label{thmn-4.3} Each left cell in $\mathcal{J}_v$ is of the form given by \eqref{thmn-4.1}
or \eqref{thmn-4.2}.
\end{enumerate}
\end{theorem}

\begin{theorem}\label{thmn-5}
Let $v\in \mathbb{N}_n$.

\begin{enumerate}[$($i$)$]
\item\label{thmn-5.1} If $v>0$, then, for each $j\in\mathbb{N}_{n-v+1}^*$, the set of all bimodules of
type $W$ and $N$ in $\mathcal{J}_v$ with left support $\{j,j+1,\dots,j+v\}$ forms a right cell in $\mathcal{J}_v$.
\item\label{thmn-5.2} For each $j\in\mathbb{N}_{n-v}^*$, the set of all bimodules of
type $M$ and $S$ in $\mathcal{J}_v$ with left support $\{j,j+1,\dots,j+v+1\}$ forms a right cell in $\mathcal{J}_v$.
\item\label{thmn-5.3} Each right cell in $\mathcal{J}_v$ is of the form given by \eqref{thmn-5.1}
or \eqref{thmn-5.2}.
\end{enumerate}
\end{theorem}

As an immediate consequence from Theorems~\ref{thmn-3}, \ref{thmn-4} and \ref{thmn-5}, we have:

\begin{corollary}\label{corn-6}
All two-sided cells in $\mathcal{J}$ are {\em strongly regular} in the sense that we have
$|\mathcal{L}\cap\mathcal{R}|=1$, for any left cell $\mathcal{L}$ in $\mathcal{J}$
and any right cell $\mathcal{R}$ in $\mathcal{J}$.
\end{corollary}

Regular two-sided cells play an important role in the theory developed in \cite{MM1}--\cite{MM6},
see also \cite{KM2}.

The algebra $A$ has an involutive anti-automorphism $\imath$ which swaps $e_1$ with $e_n$,
$e_2$ with $e_{n-1}$, $\alpha_1$ with $\alpha_{n-1}$ and so on. Using this anti-automorphism,
we can define an anti-involution on the tensor category $A$-mod-$A$ which we, abusing notation,
also denote by $\imath$. This anti-involution swaps the sides of bimodules
and twists both action by $\imath$. Consequently,
the multisemigroup $\mathcal{S}$ has an involutive anti-automorphism. This anti-automorphism makes
the statements of Theorem~\ref{thmn-4} and  Theorem~\ref{thmn-5} equivalent.
Remainder of the section is devoted to the proof of Theorems~\ref{thmn-3} and \ref{thmn-4}.

\subsection{Proof of Theorem~\ref{thmn-4}}\label{s3.28}

We start by introducing the following notation: for a subset $U\subset\mathbb{Z}$ and for $r\in\mathbb{Z}$,
we will denote by $U[r]$ the set $\{i+r\,:\, i\in U\}$.

\begin{lemma}\label{lemn-11}
Let us fixed a type, $W$, $M$, $S$ or $N$, of non-$\Bbbk$-split $A$-$A$-bimodules.
Then all bimodules of this type and with fixed right support belong to the same
left cell.
\end{lemma}

\begin{proof}
Denote by $\varphi$ the endomorphism of $A$ which sends
\begin{displaymath}
e_i\mapsto
\begin{cases}
e_{i+1},  & i+1\leq n;\\
0, & \text{ else;}
\end{cases}
\qquad
\alpha_i\mapsto
\begin{cases}
\alpha_{i+1},  & i+1< n;\\
0, & \text{ else.}
\end{cases}
\end{displaymath}
Set $\tilde{e}=e_2+e_3+\dots+e_n=\varphi(1)$. Then $\tilde{e}A$ has the natural structure of an
$A$-$A$-bimodule where the right action of $a$ is given by multiplication with $a$ and the left action
of $a$ is given by multiplication with $\varphi(a)$. We will denote this bimodule by
${}^{\varphi}\tilde{e}A$. Similarly, $A\tilde{e}$ has the natural structure of
an $A$-$A$-bimodule where the right action of $a$ is given by multiplication with $\varphi(a)$ and
the left action of $a$ is given by multiplication with $a$. We will denote this bimodule by
$A\tilde{e}^{\varphi}$.

We have $A\tilde{e}=\tilde{e}A\tilde{e}$. Consequently, the multiplication map
\begin{displaymath}
{}^{\varphi}\tilde{e}A\otimes_A A\tilde{e}^{\varphi}\to {}^{\varphi}\tilde{e}A\tilde{e}^{\varphi}
\end{displaymath}
is an isomorphism of $A$-$A$-bimodules. Set $I:=Ae_nA$. Then, mapping $1+I\mapsto \tilde{e}$,
gives rise to an isomorphism of $A$-$A$-bimodules between $A/I$ and ${}^{\varphi}\tilde{e}A\tilde{e}^{\varphi}$.
In particular, it follows that ${}^{\varphi}\tilde{e}A\otimes_A A\tilde{e}^{\varphi}\cong {}_A(A/I)_A$.

This means that, if $X\in A$-mod is annihilated by $I$, then, tensoring $X$ (from the left) first with
$A\tilde{e}^{\varphi}$ and then with ${}^{\varphi}\tilde{e}A$, gives $X$ back.
If $X$ has, additionally, the structure of an indecomposable $A$-$A$-bimodule, then tensoring with $A\tilde{e}^{\varphi}$
just twists the left action of $A$ on $X$ by $\varphi$ and hence does not change the type
($W$, $M$, $S$ or $N$) of $X$. In particular, we have
\begin{displaymath}
\mathrm{Lsupp}(A\tilde{e}^{\varphi}\otimes_A X)=(\mathrm{Lsupp}(X))[1]
\quad \text{and}\quad\mathrm{Rsupp}(A\tilde{e}^{\varphi}\otimes_A X)=\mathrm{Rsupp}(X).
\end{displaymath}
Starting now with $X$ such that $1\in\mathrm{Lsupp}(X)$ and applying this
procedure inductively, we will obtain that all bimodules of the same type as $X$ and with the same right
support as $X$ belong to the left cell of $X$. This proves the claim.
\end{proof}

\begin{lemma}\label{lemn-12}
Bimodules of types  $W$ and $S$ with the same right support belong to the same left cell.
\end{lemma}

\begin{proof}
After Lemma~\ref{lemn-11}, it is enough to prove this claim for any two particular
bimodules of types $W$ and $S$ with the same right support. Let $X$ be the unique
indecomposable subquotient of ${}_AA_A$ of type $W$ with right support $\{j,j+1,\dots,k\}$.
Assume that $k<n$ and let $Y$ be the unique
indecomposable subquotient of ${}_AA_A$ of type $S$ with right support $\{j,j+1,\dots,k\}$.
Then $Y$ surjects onto $X$ with one-dimensional kernel.

Let $Q=A(e_j+e_{j+1}+\dots+e_k)A$ be a subbimodule of
the regular bimodule ${}_AA_A$. Consider the short exact sequence
\begin{displaymath}
0\to Q\to A\to \mathrm{Coker}\to 0
\end{displaymath}
of bimodules. By construction, $\mathrm{Coker}\otimes_A X=0$ and hence
$Q\otimes_A X$ surjects onto $X$. Consequently, we either have $Q\otimes_A X\cong X$
or $Q\otimes_A X\cong Y$. To determine which of these cases takes place, we use
adjunction:
\begin{displaymath}
\mathrm{Hom}_{A\text{-}A}(Q\otimes_A X,Y)\cong
\mathrm{Hom}_{A\text{-}A}(X,\mathrm{Hom}_{A\text{-}}(Q,Y)).
\end{displaymath}
It is now easy to check that $\mathrm{Hom}_{A\text{-}}(Q,Y)\cong X$ and thus
the right hand side of the above isomorphism is non-zero.
As $\mathrm{Hom}_{A\text{-}A}(X,Y)=0$, it follows that $Q\otimes_A X=Y$.

If we denote by $Q'$ the quotient of the above bimodule $Q$ by the subbimodule $e_{k+1}Q$,
then we again have either  $Q'\otimes_A Y\cong X$ or $Q'\otimes_A Y\cong Y$. By adjunction,
\begin{displaymath}
\mathrm{Hom}_{A\text{-}A}(Q'\otimes_A Y,Y)\cong
\mathrm{Hom}_{A\text{-}A}(Y,\mathrm{Hom}_{A\text{-}}(Q',Y)).
\end{displaymath}
However, now $Q'$ is a proper quotient of $Q$ and one sees that
the dimension of $\mathrm{Hom}_{A\text{-}}(Q',Y)$ is strictly smaller than that of
$\mathrm{Hom}_{A\text{-}}(Q,Y)$. Therefore $\mathrm{Hom}_{A\text{-}}(Q',Y)$ is a proper
subbimodule of $X$ which yields
\begin{displaymath}
\mathrm{Hom}_{A\text{-}A}(Q'\otimes_A Y,Y)=0.
\end{displaymath}
This implies $Q'\otimes_A Y\cong X$ and hence $X$ and $Y$ do belong to the same left cell.

In the case $k=n$ the arguments are similar with the only difference that
$X$ has to be changed to ${}^{\varphi}\tilde{e}A\otimes_A X$. This works only
if $j>1$. If $j=1$ and $k=n$, then there are no $A$-$A$-bimodules of type $S$
with such right support.
\end{proof}

\begin{lemma}\label{lemn-13}
Bimodules of types  $M$ and $N$ with the same right support belong to the same left cell.
\end{lemma}

\begin{proof}
This is similar to the proof of Lemma~\ref{lemn-12}  and is left to the reader.
\end{proof}

\begin{lemma}\label{lemn-14}
Bimodules of types  $W$ and $M$ cannot belong to the same left cell.
\end{lemma}

\begin{proof}
Let $X$ be an $A$-$A$-bimodule of type $W$ and $Y$ an $A$-$A$-bimodule of type $M$.
Taking Lemma~\ref{lemn-1} into account,
suppose that $\mathrm{Rsupp}(X)=\mathrm{Rsupp}(Y)$ and consider the right
$A$-modules $X_A$ and $Y_A$. Then from our explicit description of bimodules we
can see that $X_A$ is not a quotient of any module in  $\mathrm{add}(Y_A)$.
Therefore no $Z\otimes_A Y_A$ can have $X_A$ as a quotient, let alone direct summand.
The claim follows.
\end{proof}

Claims~\eqref{thmn-4.1} and \eqref{thmn-4.2} of Theorem~\ref{thmn-4} follow from
Lemmata~\ref{lemn-11}--\ref{lemn-14}. Claim~\eqref{thmn-4.3} follows from
claims~\eqref{thmn-4.1} and \eqref{thmn-4.2} and classification of indecomposable
$A$-$A$-bimodules.

\subsection{Proof of Theorem~\ref{thmn-3}}\label{s3.29}

Claim~\eqref{thmn-3.1} follows from Theorems~\ref{thmn-4} and Theorem~\ref{thmn-5}.
That fact that $\mathcal{J}_v$, for $v>1$, are idempotent follows from the proof
of Lemma~\ref{lemn-12}.

To prove that two-sided cells are linearly ordered, for $j=1,2,\dots,n-1$, consider
the quotient $A$-$A$-bimodule $Q_j:=A/A(e_{j+2}+\dots+e_n)A$ of $A$ (in particular,
$Q_{n-1}=A$). We have $Q_j\in \mathcal{J}_{j}$, for all $j$, in fact,
$Q_j\cong W_{11}^{j}$. Consider also the $A$-$A$-bimodule $\mathrm{Hom}_{\Bbbk}(Q_j,\Bbbk)\cong
M_{11}^{j-1}\in \mathcal{J}_{j-1}$. We can, in fact, interpret $Q_j$ as the identity
bimodule for the algebra $B:=A/A(e_{j+2}+\dots+e_n)A$. After this interpretation it
is clear that
\begin{displaymath}
W_{11}^{j}\otimes_B M_{11}^{j-1}\cong W_{11}^{j}\otimes_A M_{11}^{j-1}\cong M_{11}^{j-1}.
\end{displaymath}
Claim~\eqref{thmn-3.2} follows.

To complete the proof of Theorem~\ref{thmn-3}, it remains to show that
$\mathcal{J}_0$ is not idempotent. Note that $\mathcal{J}_0$ only consists of
bimodules of type $M$. Let $X\otimes_A Y$ be the tensor product of two bimodules
from $\mathcal{J}_0$.
We need to check that $X\otimes_A Y$ is $\Bbbk$-split. This can be done by
a direct computation or, alternatively, argued theoretically as follows.

If $\mathrm{Rsupp}(X)\cap\mathrm{Lsupp}(Y)=\varnothing$, then $X\otimes_A Y=0$.
If $\mathrm{Rsupp}(X)=\{i,i+1\}$ and $\mathrm{Lsupp}(Y)=\{i+1,i+2\}$, then,
tensoring $Y$ with the short exact sequence
\begin{displaymath}
0\to X'\to X\to \mathrm{Coker}\to 0,
\end{displaymath}
where $Ẍ́'$ is simple with $\mathrm{Rsupp}(X')=\{i\}$ and using
$X'\otimes_A Y=0$, implies $X\otimes_AY\cong \mathrm{Coker}\otimes_A Y$ is $\Bbbk$-split
as $\mathrm{Coker}$ is $\Bbbk$-split.

If $\mathrm{Rsupp}(X)=\{i+1,i+2\}$ and $\mathrm{Lsupp}(Y)=\{i,i+1\}$, then,
tensoring $Y$ with the short exact sequence
\begin{displaymath}
0\to X'\to X\to \mathrm{Coker}\to 0,
\end{displaymath}
where $X'$ is simple with $\mathrm{Rsupp}(X')=\{i+1\}$, we obtain
$\mathrm{Coker}\otimes_A Y=0$ and $X'\otimes Y=0$
implying $X\otimes_A Y=0$ which is certainly $\Bbbk$-split.

If $\mathrm{Rsupp}(X)=\mathrm{Lsupp}(Y)=\{i,i+1\}$, then we tensor $Y$ with the short exact sequence
\begin{displaymath}
0\to X'\to X\to \mathrm{Coker}\to 0,
\end{displaymath}
where $X'$ is simple with $\mathrm{Rsupp}(X')=\{i\}$.
Then $\mathrm{Coker}\otimes_A Y=0$ and thus the two-dimensional bimodule $X'\otimes_A Y$
surjects onto $X\otimes_A Y$. Therefore $X\otimes_A Y$ has dimension at most two and thus is $\Bbbk$-split.
The proof of Theorem~\ref{thmn-3} is complete.

\subsection{Adjunctions between indecomposable $A$-$A$-bimodules}\label{s3.5}
In this subsection we will describe all pairs of bimodules which correspond to adjoint pairs of functors.

\begin{lemma}\label{lemn-21}
For any $X,Y\in A$-mod-$A$, the pair $(X\otimes_A{}_-, Y\otimes_A{}_-)$ is an adjoint pair
of endofunctors of $A$-mod if and only if  $X$ is projective
as a left $A$-module and $\mathrm{Hom}_{A}(X,A)\cong Y$ as $A$-$A$-bimodules.
\end{lemma}

\begin{proof}
As $(X\otimes_A{}_-, \mathrm{Hom}_{A}(X,{}_-))$ is an adjoint pair of functors, we have
\begin{displaymath}
Y\otimes_A{}_-\cong \mathrm{Hom}_{A}(X,{}_-),
\end{displaymath}
in particular, the latter functor must be exact and hence $X$ must be left $A$-projective.
As exact functors are uniquely determined by their action on the category of projective $A$-modules, we get
an isomorphism $\mathrm{Hom}_{A}(X,A)\cong Y$ of a $A$-$A$-bimodule. The converse implication
is straightforward.
\end{proof}

An indecomposable $A$-$A$-bimodule $X$ which is not $\Bbbk$-split is left projective
if and only if it belongs to the following set:
\begin{equation}\label{eqs}
\{W_{11}^{n-1}\}\cup\{W_{n-t\mid j}^t, S_{ij}^t, t\in\mathbb{N}_{n-1}^\ast,
i\in\mathbb{N}_{n-t}^\ast, j\in \mathbb{N}_{n-t+1}^\ast\}.
\end{equation}
As $W_{11}^{n-1}\cong{}_AA_A$, the functor $\mathrm{W}_{11}^{n-1}$ is self adjoint.
Considering the right adjoint of the functors given by tensoring with bimodules in \eqref{eqs},
we obtain the following proposition.

\begin{proposition}\label{prop3.5}
For any $t\in\mathbb{N}_{n-1}^\ast$, we have the following:
{\tiny\hspace{2mm}}
\begin{enumerate}[$($a$)$]
\item\label{prop3.5.1} for any $j\in \mathbb{N}_{n-t+1}^\ast$,
the pair $(\mathrm{W}_{n-t\vert j}^t,\mathrm{N}_{j\vert n-t-1}^t)$ is an adjoint pair of functors;
\item\label{prop3.5.2} for any $j\in \mathbb{N}_{n-t+1}^\ast$,
the pair $(\mathrm{S}_{1j}^t,\mathrm{W}_{j1}^t)$ is an adjoint pair of functors;
\item\label{prop3.5.3} for any $i\in \mathbb{N}_{n-t}^\ast\setminus\{1\}$ and $j\in \mathbb{N}_{n-t+1}^\ast$,
the pair $(\mathrm{S}_{ij}^t,\mathrm{N}_{j\vert i-1}^t)$ is an adjoint pair of functors.
\end{enumerate}
\end{proposition}

\begin{proof}
We start by proving claim~\eqref{prop3.5.1} any claim~\eqref{prop3.5.3} is proved similarly.
As $\mathrm{W}_{n-t\vert j}^t$ is left projective, its right adjoint is automatically exact.
We need to prove that $\mathrm{Hom}_{A}(\mathrm{W}_{n-t\vert j}^t,A)\cong \mathrm{N}_{j\vert n-t-1}^t$.

As $\mathrm{W}_{n-t\vert j}^t$ is indecomposable, as a bimodule, so is
$\mathrm{Hom}_{A}(\mathrm{W}_{n-t\vert j}^t,A)$. Now, recall that
indecomposable $A$-$A$-bimodules of type $N$ are uniquely determined by their left and right supports.
Therefore it is enough to check that the bimodules $\mathrm{Hom}_{A}(\mathrm{W}_{n-t\vert j}^t,A)$
and $\mathrm{N}_{j\vert n-t-1}^t$ have the same left support and the same right support.

The left $A$-action on $X:=\mathrm{Hom}_{A}(\mathrm{W}_{n-t\vert j}^t,A)$ comes from the right $A$-action on
$\mathrm{W}_{n-t\vert j}$. Because of our notation, the minimum $s$ for which $e_i$ does not annihilate
$\mathrm{W}_{n-t\vert j}$ on the right is $s=j$. Furthermore,
as a left $A$-module, $\mathrm{W}_{n-t\vert j}$ has $t+1$ direct summands. This implies that
\begin{displaymath}
\mathrm{Lsupp}(X)=\{j,j+1,\dots,j+t\}=\mathrm{Lsupp}(\mathrm{N}_{j\vert n-t-1}^t).
\end{displaymath}

The direct summand $Ae_{n-t+s}$, for $s=0,1,\dots,t$, of $\mathrm{W}_{n-t\vert j}^t$ maps only to
the direct summands $Ae_{n-t+s}$ and $Ae_{n-t+s-1}$ of $A$. This implies that
\begin{displaymath}
\mathrm{Rsupp}(X)=\{n-t-1,n-t,\dots,n\}=\mathrm{Rsupp}(\mathrm{N}_{j\vert n-t-1}^t).
\end{displaymath}
The claim follows.

Claim~\eqref{prop3.5.2} is also proved analogously to claim~\eqref{prop3.5.1} with one additional
remark: we determine the right adjoint using its support. For claim~\eqref{prop3.5.2}, the support argument
implies that the right adjoint might be either of type $W$ or of type $M$. However, exactness of
this right adjoint determines its type uniquely as $W$. The rest is completely analogous to claim~\eqref{prop3.5.1}.
\end{proof}

\section{A minimal generating set}\label{snew}

\subsection{The main result of the section}\label{snew.1}

The main result of this subsection is the following:

\begin{theorem}\label{thmw-1}
Assume $n\geq 3$.

\begin{enumerate}[$($i$)$]
\item \label{thmw-1.1}
The category $A$-mod-$A$ coincides with the minimal subcategory of $A$-mod-$A$ containing
\begin{equation}\label{eqw-1}
\{W_{11}^{n-1}\cong{}_AA_{A},\, W_{21}^{n-2},\, N_{11}^{n-2},\, S_{12}^{n-2}\}
\end{equation}
and closed under isomorphisms, tensor products and taking direct sums and direct summands.
\item\label{thmw-1.2}
The set \eqref{eqw-1} is minimal in the sense that no proper subset of
\eqref{eqw-1} has the property described in \eqref{thmw-1.1}.
\end{enumerate}
\end{theorem}

The cases $n=1,2$ are too small to be interesting.

\subsection{Comments on the proof of Lemma~\ref{lemn-11}}\label{snew.3}

Note that some of the bimodules listed in \eqref{eqw-1} appeared already in Lemma~\ref{lemn-11}.
Namely, we have $A\tilde{e}^{\varphi}\cong W_{21}^{n-2}$ and
${}^{\varphi}\tilde{e}A\cong N_{11}^{n-2}$ in $A$-mod-$A$.
We can similarly describe the $A$-$A$-bimodule $S_{12}^{n-2}$ as being
isomorphic to $A\tilde{e'}^{\psi}$, where $\tilde{e'}=e_1+e_2+\dots+e_{n-1}$
and $\psi:=\imath\circ\varphi\circ\imath$
(see Subsection~\ref{s3.27} for the definition of $\imath$). Note that
$\tilde{e'}=\psi(1)$ and that $\psi$ sends
\begin{displaymath}
e_i\mapsto
\begin{cases}
e_{i-1},  & i-1\geq 1;\\
0, & \text{ else;}
\end{cases}
\qquad
\alpha_i\mapsto
\begin{cases}
\alpha_{i-1},  & i-1\geq 1;\\
0, & \text{ else}.
\end{cases}
\end{displaymath}

Making a parallel with the proof of Lemma~\ref{lemn-11}, we also have
the $A$-$A$-bimodule ${}^{\psi}\tilde{e'}A$. This bimodule is isomorphic to $W_{21}^{n-2}$.
In fact, the anti-involution $\imath$ on $A$-mod-$A$
fixes both $W_{11}^{n-1}$ and $W_{21}^{n-2}$ but swaps $N_{11}^{n-2}$ with $S_{12}^{n-2}$.
Similarly to the proof of Lemma~\ref{lemn-11}, we have
\begin{displaymath}
{}^{\psi}\tilde{e'}A\otimes_A A\tilde{e'}^{\psi}\cong {}^{\psi}\tilde{e'}A\tilde{e'}^{\psi}\cong{}_A(A/I')_A,
\end{displaymath}
where $I'=Ae_1A=Ae_1$.
This means that, if $X\in$ mod-$A$ is annihilated by $I'$, then,
tensoring $X$ first with
${}^{\psi}\tilde{e'}A$ and then with $A\tilde{e'}^{\psi}$ (both from the right), gives $X$ back.
If $X$ has the structure of an indecomposable $A$-$A$-bimodule, then tensoring with ${}^{\psi}\tilde{e'}A$
just twists the right action of $A$ on $X$ by $\psi$ and hence does not change the type
($W$, $M$, $S$ or $N$) of $X$.
Moreover, we have
\begin{displaymath}
\mathrm{Lsupp}(X\otimes_A {}^{\psi}\tilde{e'}A)=\mathrm{Lsupp}(X) \quad \text{and}\quad
\mathrm{Rsupp}(X\otimes_A {}^{\psi}\tilde{e'}A)=(\mathrm{Rsupp}(X))[-1]
\end{displaymath}
(here we use the notation $U[-1]$ from Subsection~\ref{s3.28}).
This fact can be used to prove the ``other side'' version of~Lemma~\ref{lemn-11},
which consequently contributes to the proof of~Theorem~\ref{thmn-5}.

\subsection{Auxiliary lemmata}\label{snew.2}

By construction, the action graph $\Gamma_X$ of an $A$-$A$-bi\-mo\-du\-le $X$ is a subgraph
of the graph \eqref{eq0}. In what follows, for a fixed $A$-$A$-bimodule $X$, we will describe the
outcome of tensoring of $X$ with $W_{21}^{n-2}$, $N_{11}^{n-2}$, and $S_{12}^{n-2}$, both from
the left and from the right, using combinatorial manipulations with the graph $\Gamma_X$,
considering the latter as a subgraph of \eqref{eq0}.

\begin{lemma}\label{lemw-1}
Let $X$ be an indecomposable $A$-$A$-bimodule.

\begin{enumerate}[$($i$)$]
\item\label{lemw-1.1} The action graph of the $A$-$A$-bimodule $W_{21}^{n-2}\otimes_A X$
is obtained from $\Gamma_X$ by moving the latter vertically one step down and then cutting
off all vertices and arrows which fall outside the graph in figure \eqref{eq0}.
\item\label{lemw-1.2} The action graph of the $A$-$A$-bimodule $X\otimes_A W_{21}^{n-2}$
is obtained from $\Gamma_X$ by moving the latter horizontally one step to the left
and then cutting off all vertices and arrows which fall outside the graph  in figure \eqref{eq0}.
\end{enumerate}
\end{lemma}

\begin{proof}
From the proof of Lemma~\ref{lemn-11}, it follows that claim~\eqref{lemw-1.1} is true
as soon as $IX=0$ (recall that $I=Ae_nA=e_nA$). Note that in this case, no ``cutting off''
is necessary.

In general, we have $W_{21}^{n-2}\cong A\tilde{e}^{\varphi}$
and $IX=e_n X$ is an $A$-$A$-subbimodule of $X$.
If $e_nX\neq 0$, then we have $A\tilde{e}^{\varphi}\otimes_A e_nX=0$ and
hence $A\tilde{e}^{\varphi}\otimes_A X\cong A\tilde{e}^{\varphi}\otimes_A (X/e_nX)$.
As $I(X/e_nX)=0$, we can apply the argument from the previous paragraph.
Factoring $e_nX$ out corresponds precisely to ``cutting off'' those vertices and edges
which fall outside the graph  in figure \eqref{eq0} after the move. This completes the proof of claim~\eqref{lemw-1.1}.

From Subsection~\ref{snew.3} it follows that claim~\eqref{lemw-1.2} is true
as soon as  $XI'=0$. Note that in this case, no ``cutting off''
is necessary.

In general, we have $W_{21}^{n-2}\cong {}^{\psi}\tilde{e'}A$ and $XI'=Xe_1$
is an $A$-$A$-subbimodule of $X$.
If $Xe_1\neq 0$, then we have $Xe_1\otimes_A{}^{\psi}\tilde{e'}A=0$ and hence
\begin{displaymath}
X\otimes_A{}^{\psi}\tilde{e'}A\cong (X/Xe_1)\otimes_A{}^{\psi}\tilde{e'}A .
\end{displaymath}
As  $(X/Xe_1)I'=0$, we can apply the argument from the previous paragraph.
Factoring $Xe_1$ out corresponds precisely to ``cutting off'' those vertices and edges
which fall outside the graph  in figure \eqref{eq0} after the move.
This completes the proof of claim~\eqref{lemw-1.2}.
\end{proof}

Let $X$ be an indecomposable $A$-$A$-bimodule. A full subgraph $\Gamma$ of $\Gamma_X$ will be
called {\em thick} provided that, for any arrow $x\to y$ in $\Gamma_X$, the condition
$x\in \Gamma$ implies $y\in\Gamma$.

\begin{lemma}\label{lemw-2}
Let $X$ be an indecomposable and not $\Bbbk$-split $A$-$A$-bimodule.

\begin{enumerate}[$($i$)$]
\item\label{lemw-2.1} The action graph of the $A$-$A$-bimodule $N_{11}^{n-2}\otimes_A X$
is obtained from $\Gamma_X$ by moving the latter  vertically one step up
and then cutting off all vertices and edges which fall outside the graph  in figure  \eqref{eq0}.
\item\label{lemw-2.2} If $\dim_{\Bbbk} Xe_1=2$, then the action graph of the
$A$-$A$-bimodule $X\otimes_A N_{11}^{n-2}$ is obtained from $\Gamma_X$ in the following three steps:
\begin{itemize}
\item first we move $\Gamma_X$ one step to the right;
\item then we cut off the thick subgraph generated by
all vertices which fall outside the graph  in figure  \eqref{eq0}, we denote the resulting graph $\Gamma$;
\item finally, we add to $\Gamma$ one new vertex and one new arrow as follows: let $v$ be the
north-west corner of $\Gamma$, then we add to $\Gamma$ the immediate west neighbor $w$ of $v$
and the arrow connecting $v$ to $w$.
\end{itemize}
\item \label{lemw-2.3} If $\dim_{\Bbbk} Xe_1\neq2$, then the action graph of the $A$-$A$-bimodule
$X\otimes_A N_{11}^{n-2}$ is obtained from $\Gamma_X$ by moving the latter horizontally one step to the right
and then cutting off the thick subgraph generated by
all vertices which fall outside the graph in figure \eqref{eq0}.
\end{enumerate}
\end{lemma}

A good example to illustrate the procedure described in Lemma~\ref{lemw-2}\eqref{lemw-2.2} is
the obvious isomorphism $W_{11}^{n-1}\otimes_A N_{11}^{n-2}\cong N_{11}^{n-2}$ based on the fact that
$W_{11}^{n-1}\cong {}_AA_A$.
In particular, this example shows that the second step of the procedure described in
Lemma~\ref{lemw-2}\eqref{lemw-2.2} can lead to elimination of some vertices which do not fall outside
the graph  in figure \eqref{eq0}. For $n=4$, the transformation in this example
can be depicted explicitly as follows:
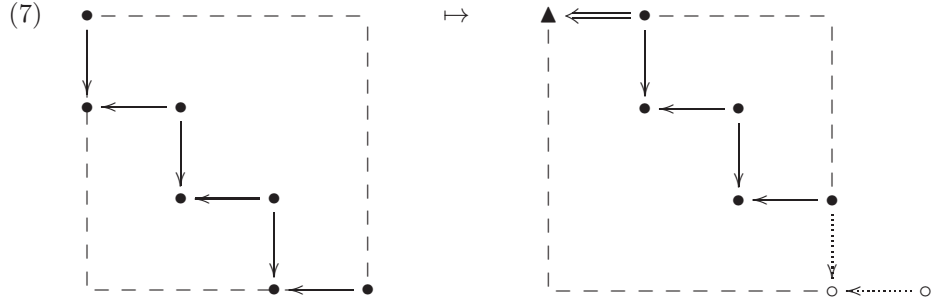
\begin{equation}\label{eqeqnn1}
\xymatrix{
\bullet\ar[d]\ar@{--}[ddd]\ar@{--}[rrr]&&&\ar@{--}[ddd]\\
\bullet&\bullet\ar[d]\ar[l]&&\\
&\bullet&\bullet\ar[d]\ar[l]&\\
\ar@{--}[rrr]&&\bullet&\bullet\ar[l]\\
}\qquad\mapsto\qquad
\xymatrix{
\ar@{--}[ddd]\blacktriangle&\bullet\ar[d]\ar@{=>}[l]\ar@{--}[rr]&&\ar@{--}[dd]&\\
&\bullet&\bullet\ar[d]\ar[l]&&\\
&&\bullet&\bullet\ar@{.>}[d]\ar[l]&\\
\ar@{--}[rrr]&&&\circ&\circ\ar@{.>}[l]\\
}
\end{equation}
Here dashed lines indicate the boundaries of the graph  in figure \eqref{eq0},
dotted arrows and white vertices are the ones which are deleted during the second step of the
procedure described in  Lemma~\ref{lemw-2}~\eqref{lemw-2.2}. Finally, the triangle
vertex and the double arrow are the ones which are added during the last step of the
procedure described in  Lemma~\ref{lemw-2}~\eqref{lemw-2.2}.

\begin{proof}
Assume that $e_1X=0$ and let $Y$ be an $A$-$A$-bimodule such that $\Gamma_Y$ is obtained
by moving  $\Gamma_X$ one step up  (note that this is well-defined as $e_1X=0$). Then
$Y$ has the same type as $X$ and
\begin{displaymath}
\mathrm{Lsupp}(Y)=\mathrm{Lsupp}(X)[-1] \quad \text{and}\quad \mathrm{Rsupp}(Y)=\mathrm{Rsupp}(X).
\end{displaymath}
From Lemma~\ref{lemw-1}~\eqref{lemw-1.1} it follows that
$X\cong W_{21}^{n-2}\otimes_A Y$. This, together with the discussion in Subsection~\ref{snew.3},
implies $N_{11}^{n-2} \otimes_A X\cong Y$ in $A$-mod-$A$.
Claim~\eqref{lemw-2.1} in the case $e_1X=0$ follows.

Assume now that $e_1X\neq 0$.
Consider the following short exact sequence of $A$-$A$-bimodules:
\begin{equation}\label{neweq7}
0\to \tilde{e}X\to X\to \mathrm{Coker}\to 0,
\end{equation}
where $\tilde{e}X\to X$ is the natural inclusion. In this sequence we have that
${}_A\mathrm{Coker}$ is  semisimple, moreover, $e_i\mathrm{Coker}=0$, for all $i>1$.
Recall that $N_{11}^{n-2}\cong{}^{\varphi}\tilde{e}A$.
Then ${}^{\varphi}\tilde{e}A\otimes_A \mathrm{Coker}=0$. The functor
$N_{11}^{n-2}\otimes_A{}_-$ is exact as $N_{11}^{n-2}$ is right projective
(or due to Proposition~\ref{prop3.5}~\eqref{prop3.5.1}).
Therefore, applying
$N_{11}^{n-2}\otimes_A {}_-$ to \eqref{neweq7},
gives $N_{11}^{n-2}\otimes_A \tilde{e}X\cong N_{11}^{n-2}\otimes_AX$.
We can now apply the previous paragraph to $N_{11}^{n-2}\otimes_A \tilde{e}X$
and thus complete the proof of claim~\eqref{lemw-2.1}.

We have $\mathrm{Lsupp}(X\otimes_A {}^{\varphi}\tilde{e}A)\subset\mathrm{Lsupp}(X)$
by Lemma~\ref{lemn-1}.
Similarly to the proof of Lemma~\ref{lemn-11},
if $XI=0$, then, tensoring $X$ (from the right) with
${}^{\varphi}\tilde{e}A$ and then with $A\tilde{e}^{\varphi}$, gives $X$ back.
Consequently, in the case  $XI=0$ we have
\begin{displaymath}
\mathrm{Lsupp}(X\otimes_A {}^{\varphi}\tilde{e}A)=\mathrm{Lsupp}(X).
\end{displaymath}

Consider the short exact sequence
\begin{equation}\label{neweq8}
0\to {}^{\varphi}\tilde{e}Ae_1\to {}^{\varphi}\tilde{e}A\to \mathrm{Coker'}\to 0
\end{equation}
of $A$-$A$-bimodules where ${}^{\varphi}\tilde{e}Ae_1\to {}^{\varphi}\tilde{e}A$
is the natural inclusion. Then we have  $\mathrm{Coker'}\cong W_{12}^{n-2}$.
Using Lemma~\ref{lemw-1}~\eqref{lemw-1.2} or the fact that $e_1A\tilde{e}^{\varphi}=0$, we have
\begin{displaymath}
\mathrm{Coker'}\otimes_AA\tilde{e}^{\varphi}\cong \mathrm{{}^{\varphi}\tilde{e}A}\otimes_A A\tilde{e}^{\varphi}.
\end{displaymath}
Therefore, tensoring $X$ (from the right) with $\mathrm{Coker'}$ and then with
$A\tilde{e}^{\varphi}$ also gives $X$ back.
Applying the functor $X\otimes_A{}_-$ to~\eqref{neweq8}, we thus obtain
\begin{displaymath}
\mathrm{Rsupp}(X\otimes_A {}^{\varphi}\tilde{e}A)\subset
\mathrm{Rsupp}(X\otimes_A {}^{\varphi}\tilde{e}Ae_1)\cup
\mathrm{Rsupp}(X\otimes_A \mathrm{Coker'}).
\end{displaymath}
As tensor functors are right exact, we also have
\begin{displaymath}
\mathrm{Rsupp}(X\otimes_A \mathrm{Coker'})\subset\mathrm{Rsupp}(X\otimes_A {}^{\varphi}\tilde{e}A).
\end{displaymath}
Just like in the proof of Lemma~\ref{lemn-11},
tensoring with $W_{12}^{n-2}$ twists the right $A$-action on $X$ by $\varphi$
and hence does not change the type ($W$, $M$, $S$, or $N$) of $X$.
In particular, we have
\begin{displaymath}
\mathrm{Lsupp}(X\otimes_A \mathrm{Coker'})=\mathrm{Lsupp}(X)\quad\text{and}\quad\mathrm{Rsupp}(X\otimes_A \mathrm{Coker'})=\mathrm{Rsupp}(X)[1].
\end{displaymath}
Therefore, considering the action graph of $X\otimes_A \mathrm{Coker'}$ corresponds to
the first step of claim~\eqref{lemw-2.2}.

Note that $\mathrm{Rsupp}(X\otimes_A {}^{\varphi}\tilde{e}Ae_1)\subset \{1\}$.
Hence, the above arguments imply
\begin{displaymath}
(\mathrm{Rsupp}(X))[1]\subset\mathrm{Rsupp}(X\otimes_A {}^{\varphi}\tilde{e}A)\subset\{1\}\cup(\mathrm{Rsupp}(X))[1].
\end{displaymath}

Now, if $\dim_{\Bbbk} Xe_1=2$, then~$\mathbf{B}(X)$ contains the vertices $\mathtt{j\mid 1}$
and  $\mathtt{j\text{+}1\mid 1}$, for some $j$.
Moreover, we obtain that the $A$-$A$-bimodule $X\otimes_A {}^{\varphi}\tilde{e}Ae_1$ is of dimension one with
basis $\mathtt{j\mid 1}\otimes \alpha_1$ (here it is important that 
$X$ is not $\Bbbk$-split). The image of this element in
$X\otimes_A {}^{\varphi}\tilde{e}A$ is also non-zero.
This implies that
\begin{displaymath}
\mathrm{Rsupp}(X\otimes_A {}^{\varphi}\tilde{e}A)=\{1\}\cup(\mathrm{Rsupp}(X))[1].
\end{displaymath}
Therefore the sequence
\begin{displaymath}
0\to X\otimes_A{}^{\varphi}\tilde{e}Ae_1\to X\otimes_A{}^{\varphi}\tilde{e}A\to X\otimes_A\mathrm{Coker'}\to 0
\end{displaymath}
is exact, which completes the proof of claim~\eqref{lemw-2.2} in the case $XI=0$.

If $\dim_{\Bbbk} Xe_1\neq2$,
then we have $X\otimes_A{}^{\varphi}\tilde{e}Ae_1=0$ as ${}_A({}^{\varphi}\tilde{e}Ae_1)$ is 
simple and $X$ is not $\Bbbk$-split.
This implies that $X\otimes_A{}^{\varphi}\tilde{e}A\cong X\otimes_A\mathrm{Coker'}$, which
establishes claim~\eqref{lemw-2.3} in the case $XI=0$.

If $XI\neq0$, then we have $XI\otimes_A {}^{\varphi}\tilde{e}A=0$
as the left action of $e_n$ annihilates ${}^{\varphi}\tilde{e}A$.
Hence we obtain $X\otimes_A {}^{\varphi}\tilde{e}A\cong (X/XI)\otimes_A {}^{\varphi}\tilde{e}A$
so that we can reduce our consideration to the previous case $XI=0$. This proves
claims~\eqref{lemw-2.2} and \eqref{lemw-2.3} in full generality and we are done.
\end{proof}
\vspace{1cm}

\begin{lemma}\label{lemw-3}
Let $X$ be an indecomposable and not  $\Bbbk$-split $A$-$A$-bimodule.

\begin{enumerate}[$($i$)$]
\item\label{lemw-3.1} If $\dim e_nX=2$, then the action graph of the
$A$-$A$-bimodule $S_{12}^{n-2}\otimes_A X$ is obtained from $\Gamma_X$ in the following three steps:
\begin{itemize}
\item first we move $\Gamma_X$ vertically one step up;
\item then we cut off the thick subgraph generated by
all vertices which fall outside the graph  in figure  \eqref{eq0}, we denote the resulting graph $\Gamma$;
\item finally, we add to $\Gamma$ one new vertex and one new arrow as follows: let $v$ be the
south-east corner of $\Gamma$, then we add to $\Gamma$ the immediate south neighbor $w$ of $v$
and the arrow connecting $v$ to $w$.
\end{itemize}
\item\label{lemw-3.2} If $\dim e_n X\neq 2$, then the action graph of the
$A$-$A$-bimodule $S_{12}^{n-2}\otimes_A X$ is obtained from $\Gamma_X$ by
moving the latter vertically one step up and then cutting off the thick subgraph generated by
all vertices which fall outside the graph in figure \eqref{eq0}.
\item\label{lemw-3.3} The action graph of the $A$-$A$-bimodule
$X\otimes_A S_{12}^{n-2}$ is obtained from $\Gamma_X$ by moving the latter
one step to the right and then cutting off all vertices and edges which fall
outside the graph  in figure  \eqref{eq0}.
\end{enumerate}
\end{lemma}

\begin{proof}
Observing that we have an isomorphism $S_{12}^{n-2}\cong A\tilde{e'}^{\psi}$ of $A$-$A$-bimodules,
the proof is similar to that of Lemma~\ref{lemw-2} and is left to the reader.
\end{proof}

Again, a good example to illustrate the procedure described in Lemma~\ref{lemw-3}~\eqref{lemw-3.1} is
the obvious isomorphism $S_{12}^{n-2}\otimes_A W_{11}^{n-1}\cong S_{12}^{n-2}$ based on the fact that
$W_{11}^{n-1}\cong {}_AA_A$.
In particular, this example shows that the second step of the procedure described in
Lemma~\ref{lemw-3}~\eqref{lemw-3.1} can lead to elimination of some vertices which do not fall outside
the graph  in figure \eqref{eq0}. For $n=4$, the transformation in this example
can be depicted (using the same conventions as in \eqref{eqeqnn1}) explicitly as follows:
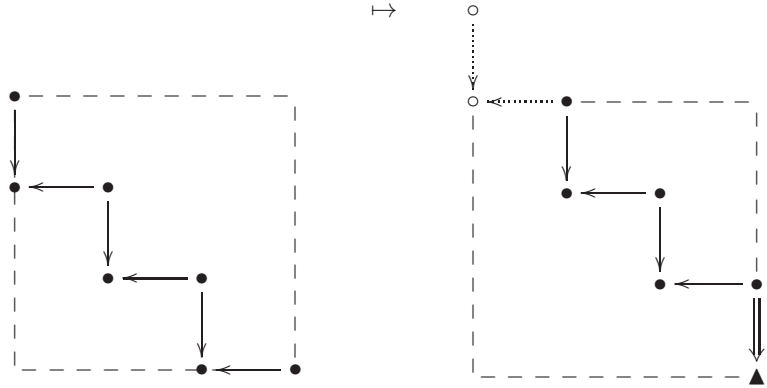
\begin{displaymath}
\xymatrix{
&&&\\
\bullet\ar[d]\ar@{--}[ddd]\ar@{--}[rrr]&&&\ar@{--}[ddd]\\
\bullet&\bullet\ar[d]\ar[l]&&\\
&\bullet&\bullet\ar[d]\ar[l]&\\
\ar@{--}[rrr]&&\bullet&\bullet\ar[l]\\
}\qquad\mapsto\qquad
\xymatrix{
\circ\ar@{.>}[d]&&&\\
\ar@{--}[ddd]\circ&\bullet\ar[d]\ar@{.>}[l]\ar@{--}[rr]&&\ar@{--}[dd]&\\
&\bullet&\bullet\ar[d]\ar[l]&&\\
&&\bullet&\bullet\ar@{=>}[d]\ar[l]&\\
\ar@{--}[rrr]&&&\blacktriangle\\
}
\end{displaymath}

\subsection{Proof of Theorem~\ref{thmw-1}}\label{snew.5}

Let $\mathcal{X}$ denote the minimal subcategory of $A$-mod-$A$ containing \eqref{eqw-1}
and closed under isomorphisms, tensor products and taking direct sums and direct summands.
To prove Theorem~\ref{thmw-1}~\eqref{thmw-1.1}, we need to show that all indecomposable
$A$-$A$-bimodules belong to $\mathcal{X}$.

Clearly, $W_{11}^{n-1}\in \mathcal{X}$ as it appears in \eqref{eqw-1}. Further,
$M_{11}^{n-2}\in \mathcal{X}$ as, for example,
\begin{displaymath}
S_{12}^{n-2}\otimes_A ( W_{21}^{n-2}\otimes_{A}N_{11}^{n-2})
\cong S_{12}^{n-2}\otimes_A N_{21}^{n-2}\cong M_{11}^{n-2},
\end{displaymath}
where we used Lemma~\ref{lemw-2}~\eqref{lemw-2.2} for the first
isomorphism and Lemma~\ref{lemw-3}~\eqref{lemw-3.1} for the second one.

Starting from $W_{11}^{n-1}$, $M_{11}^{n-2}$, $S_{12}^{n-2}$ and $N_{11}^{n-2}$, and applying
the manipulations described in Lemma~\ref{lemw-1}~\eqref{lemw-1.1}-\eqref{lemw-1.2},
Lemma~\ref{lemw-2}~\eqref{lemw-2.1} and Lemma~\ref{lemw-3}~\eqref{lemw-3.3}, it is clear that
we can obtain action graphs of all indecomposable $A$-$A$-bimodules which are not isomorphic to
$Ae_i\otimes_{\Bbbk}e_{j+1}A$, where  $i,j\in\mathbb{N}_n^\ast$.
Hence all such bimodules must belong to $\mathcal{X}$, in particular, all
indecomposable string $A$-$A$-bimodules of dimension two are in $\mathcal{X}$.

Finally, tensoring indecomposable string $A$-$A$-bimodules of dimension two with each other,
we can get all bimodules of the form $Ae_i\otimes_{\Bbbk}e_{j+1}A$, where  $i,j\in\mathbb{N}_n^\ast$.
Claim~\eqref{thmw-1.1} of Theorem~\ref{thmw-1} follows.

Now, let us prove claim~\eqref{thmw-1.2}. Let $T$ be a proper subset of~\eqref{eqw-1} and
$\mathcal{X}_T$ the minimal subcategory of $A$-mod-$A$ containing $T$ and closed under
isomorphisms, tensor products and taking direct sums and direct summands. We have to show that
$\mathcal{X}_T$ is a proper subcategory of $A$-mod-$A$. If $W_{11}^{n-1}\not\in T$,
then the latter claim follows directly from Theorem~\ref{thmn-3}~\eqref{thmn-3.2}.

If $W_{21}^{n-2}\not\in T$, then, from Lemmata~\ref{lemw-2} and \ref{lemw-3} it follows
that the only possible manipulations with actions graphs are to move them up or to the
right. As we can only start with action graphs of bimodules from $T$, it follows
that, in particular,  $W_{21}^{n-2}\not\in \mathcal{X}_T$. A similar argument also shows that,
necessarily, either $S_{12}^{n-2}\in T$ or $N_{11}^{n-2}\in T$.

Assume that $T$ is \eqref{eqw-1} without $S_{12}^{n-2}$. Then, from Lemmata~\ref{lemw-1} and \ref{lemw-2},
it follows that all bimodules in $\mathcal{X}_T$ are either $\Bbbk$-split or of type
$W$ or $N$.
Assume that $T$ is \eqref{eqw-1} without $N_{11}^{n-2}$. Then, from Lemmata~\ref{lemw-1} and \ref{lemw-3},
it follows that all bimodules in $\mathcal{X}_T$ are either $\Bbbk$-split or of type
$W$ or $S$. Claim~\eqref{thmw-1.2} follows and the proof of Theorem~\ref{thmw-1}
is complete.

\section{Simple transitive $2$-representations of projective $A$-$A$-bimodules}\label{s4}

\subsection{Finitary $2$-categories and their $2$-representations}\label{s4.1}
In this section we switch from the concrete algebras $A_n$ considered above to
general finite dimensional algebras $A$.

For a finite dimensional $\Bbbk$-algebra $A$, consider the $2$-category $\cC_A$ of projective
endofunctors of $A$-mod, see \cite[Subsection~7.3]{MM1} (we note that this $2$-category depends
on the choice of a small category equivalent to $A$-mod). We assume that $A$ is basic and connected
and let $\varepsilon_1+\varepsilon_2+\dots+\varepsilon_k=1$ be a primitive decomposition of the
identity in $A$. The $2$-category $\cC_A$ has one object $\mathtt{i}$. A complete list of
representatives of isomorphism classes of indecomposable $A$-$A$-bimodules which contribute to
$1$-morphisms in $\cC_A$ consists of the regular bimodule ${}_AA_A$, which corresponds to the
identity $1$-morphism $\mathbbm{1}_{\mathtt{i}}$, and indecomposable projective bimodules $A\varepsilon_i\otimes\varepsilon_jA$, where $i,j=1,2,\dots,k$, which
correspond to indecomposable $1$-morphisms respectively denoted by $\mathrm{F}_{ij}$. Finally,
$2$-morphisms in $\cC_A$ are given by homomorphisms of $A$-$A$-bimodules.

A {\em finitary $2$-representation} of $\cC_A$ is a functorial action, denoted $\mathbf{M}$,
on a category $\mathbf{M}(\mathtt{i})$ equivalent to $B$-proj of projective modules over
some finite dimensional $\Bbbk$-algebra $B$. All such $2$-representations form a $2$-category,
denoted $\cC_A$-afmod, where $1$-morphisms are $2$-natural transformations and $2$-morphisms
are modifications, see \cite{MM3} for details.

A finitary $2$-representation $\mathbf{M}$ is called {\em transitive} if $\mathbf{M}(\mathtt{i})$
has no proper $\cC_A$-invariant, idempotent split and isomorphism closed additive subcategories.
A transitive $2$-representation $\mathbf{M}$ is called {\em simple} if $\mathbf{M}(\mathtt{i})$
has no proper $\cC_A$-invariant ideals, see \cite{MM5,MM6} for details.

Classical examples of simple transitive $2$-representations are so-called {\em cell $2$-rep\-re\-sen\-ta\-ti\-ons}
as defined in \cite{MM1,MM2}. The $2$-category $\cC_A$ has, up to equivalence, two cell $2$-representations:
\begin{itemize}
\item The cell $2$-representation $\mathbf{C}_{\{\mathbbm{1}_{\mathtt{i}}\}}$ which is given as the quotient of the
left regular action of $\cC_A$ on $\cC_A(\mathtt{i},\mathtt{i})$ by the unique maximal
$\cC_A$-invariant left ideal.
\item The cell $2$-representation $\mathbf{C}_{\{\mathrm{F}_{i1}\}}$ which is given 
(up to equivalence) by the defining action of $\cC_A$ on $A$-proj.
\end{itemize}

\subsection{The main result of the section}\label{s4.2}

The main result of this section is the following:

\begin{theorem}\label{thmmain3}
Assume that $A$ has a non-zero projective injective module and
is directed in the sense that $\varepsilon_i A\varepsilon_j=0$ whenever $i<j$
and $\varepsilon_i A\varepsilon_i=\Bbbk \varepsilon_i$, for all $i$.
Then every simple transitive $2$-representation of $\cC_A$ is equivalent to a cell $2$-representation.
\end{theorem}

The algebra $A_n$ from Subsection~\ref{s2.1} obviously satisfies both assumption of Theorem~\ref{thmmain3}.
Therefore Theorem~\ref{thmmain3} provides a classification of simple
transitive $2$-representation of $\cC_{A_n}$.
In fact, any quotient of the path algebra of the quiver \eqref{eq1} satisfies both assumption of
Theorem~\ref{thmmain3}. There are of course many other algebras which satisfy these assumptions,
for example incidence algebras of finite posets having the minimum and the maximum element
(for example, the Boolean of a finite set) and many others.
In the cases $A=A_2$ and $A=A_3$, Theorem~\ref{thmmain3} is proved in \cite{MZ}.

The rest of this section is devoted to the proof of Theorem~\ref{thmmain3}.

\subsection{Notational preparation}\label{s4.3}

We let $\mathbf{M}$ be a simple transitive $2$-representation of $\cC_A$ and denote by
$B$ a basic $\Bbbk$-algebra such that $\mathbf{M}(\mathtt{i})$ is equivalent to $B$-proj.
Let $\epsilon_1+\epsilon_2+\dots+\epsilon_r=1$ be a primitive decomposition of the
identity in $B$. For $i,j=1,2,\dots,r$, we denote by $\mathrm{G}_{ij}$ the endofunctor
of $B$-mod given by tensoring with the indecomposable projective $B$-$B$-bimodule
$B\epsilon_i\otimes \epsilon_jB$. We note that $r\neq k$, in general.

Without loss of generality we may assume that $\mathbf{M}$ is faithful. Indeed,
the $2$-category $\cC_A$ is simple, see \cite[Subsection~3.2]{MMZ}. Therefore, if $\mathbf{M}$
is not faithful, then $\mathbf{M}(\mathrm{F}_{ij})=0$, for all $i,j$. The quotient of
$\cC_A$ by the $2$-ideal generated by all $\mathrm{F}_{ij}$ satisfies the assumptions of
\cite[Theorem~18]{MM5} and hence $\mathbf{M}$ is equivalent to a cell $2$-representation by
\cite[Theorem~18]{MM5}.

So, from now on, $\mathbf{M}$ is faithful, in particular, each $\mathbf{M}(\mathrm{F}_{ij})$ is non-zero.
As $A$ has a non-zero projective-injective module, by \cite[Section~3]{MZ}, each
$\mathbf{M}(\mathrm{F}_{ij})$ is a projective endofunctor of $B$-mod, that is, is isomorphic
to a direct sum of some $\mathrm{G}_{st}$, $s,t\in\{1,2,\dots,r\}$, possibly even with multiplicities.
For $i,j=1,2,\dots,k$, we denote by
\begin{itemize}
\item $X_{ij}$ the set of all $s\in\{1,2,\dots,r\}$ such that $\mathrm{G}_{st}$ is isomorphic to a
direct summand of $\mathbf{M}(\mathrm{F}_{ij})$, for some $t\in\{1,2,\dots,r\}$;
\item $Y_{ij}$ the set of all $t\in\{1,2,\dots,r\}$ such that $\mathrm{G}_{st}$ is isomorphic to a
direct summand of $\mathbf{M}(\mathrm{F}_{ij})$, for some $s\in\{1,2,\dots,r\}$.
\end{itemize}
As each $\mathbf{M}(\mathrm{F}_{ij})$ is non-zero, each $X_{ij}$ and each $Y_{ij}$ is not empty.

Recall that $A$ is assumed to have a non-zero projective-injective module. Then there exist
$i_0,j_0\in\{1,2,\dots,k\}$ such that $A\varepsilon_{i_0}\cong\mathrm{Hom}_\Bbbk(\varepsilon_{j_0}A,\Bbbk)$.
In this case, for every $q\in \{1,2,\dots,k\}$, the pair
\begin{equation}\label{eq4.2}
(\mathrm{F}_{qj_0},\mathrm{F}_{i_0q})
\end{equation}
is an adjoint pair of $1$-morphisms, see \cite[Lemma~5]{MZ}.

For $s=1,2,\dots,r$, we denote by $P_s$ the indecomposable projective $B$-module $B\epsilon_s$
and by $L_s$ the simple top of $P_s$. Whenever it does not lead to any confusion, we will use
{\em action notation} and simply write e.g.  $\mathrm{F}_{ij}\, M$, for $M\in \mathbf{M}(\mathtt{i})$
and $M\in B$-mod, instead of $\mathbf{M}(\mathrm{F}_{ij})(M)$.

\subsection{Analysis of the sets $X_{ij}$ and $Y_{ij}$}\label{s4.4}

\begin{lemma}\label{lem4.1}
For $i=1,2,\dots,k$, we have $X_{ij_1}=X_{ij_2}$, for all $j_1,j_2\in\{1,2,\dots,k\}$.
Similarly, for $j=1,2,\dots,k$, we have $Y_{i_1j}=Y_{i_2j}$, for all $i_1,i_2\in\{1,2,\dots,k\}$.
\end{lemma}

\begin{proof}
We prove the first claim, the proof of the second one is similar. We have
\begin{displaymath}
\big(A\varepsilon_i\otimes  \varepsilon_{j_1}A\big)\otimes_A
\big(A\varepsilon_{j_1}\otimes  \varepsilon_{j_2}A\big)\cong
A\varepsilon_i\otimes  \varepsilon_{j_2}A ^{\oplus \dim(\varepsilon_{j_1}A\varepsilon_{j_1})}.
\end{displaymath}
Note that $\dim(\varepsilon_{j_1}A\varepsilon_{j_1})>0$. Therefore
$\mathrm{F}_{ij_2}$ is isomorphic to a direct summand of $\mathrm{F}_{ij_1}\circ \mathrm{F}_{j_1j_2}$.
This implies $X_{ij_2}\subset X_{ij_1}$ (as multiplication on the right cannot create new indexing
idempotents on the left). By symmetry, we also have  $X_{ij_1}\subset X_{ij_2}$. The claim follows.
\end{proof}

After Lemma~\ref{lem4.1}, for $i=1,2,\dots,k$, we may denote by $X_i$ the common value of all
$X_{ij}$, where $j\in\{1,2,\dots,k\}$. Similarly, for $j=1,2,\dots,k$, we may denote by $Y_j$
the common value of all $Y_{ij}$, where $i\in\{1,2,\dots,k\}$.

\begin{lemma}\label{lem4.2}
We have $X_1\cup X_2\cup\dots\cup X_k=\{1,2,\dots,r\}$.
\end{lemma}

\begin{proof}
The set $X_1\cup X_2\cup\dots\cup X_k$ indexes those projectives that can be obtained using
the action of $\cC_A$. Therefore the claim follows immediately from transitivity of $\mathbf{M}$.
\end{proof}

\begin{lemma}\label{lem4.3}
For every $q=1,2,\dots,k$, we have  $X_q=Y_q$.
\end{lemma}

\begin{proof}
Consider the pair $(\mathrm{F}_{qj_0},\mathrm{F}_{i_0q})$ of adjoint $1$-morphisms given by \eqref{eq4.2}.
By adjunction, for $s\in\{1,2,\dots,r\}$, we have
\begin{displaymath}
\mathrm{Hom}_{B}(\mathrm{F}_{qj_0}\, B, L_s)\cong \mathrm{Hom}_{B}(B, \mathrm{F}_{i_0q}\, L_s),
\end{displaymath}
in particular, the left hand side is non-zero if and only if the right hand side is non-zero.
At the same time, the left hand side is non-zero if and only if $P_s$ appears in
$\mathrm{F}_{qj_0}\, B$, that is, $s\in X_q$; while the right hand side is non-zero
if and only if $\mathrm{F}_{i_0q}\, L_s\neq 0$, that is, $s\in Y_q$. The claim follows.
\end{proof}

As an immediate consequence of Lemmata~\ref{lem4.2} and \ref{lem4.3}, we have:

\begin{corollary}\label{cor4.4}
We have $Y_1\cup Y_2\cup\dots\cup Y_k=\{1,2,\dots,r\}$.
\end{corollary}

Note that, until this point, we never used that $A$ is directed. It will, however, be
crucial for the following two lemmata.

\begin{lemma}\label{lem4.5}
For $i_1\neq i_2\in \{1,2,\dots,k\}$, we have $X_{i_1}\cap X_{i_2}=\varnothing$.
\end{lemma}

\begin{proof}
Let $s\in X_{i_1}\cap X_{i_2}$. Then $s\in Y_{i_1}\cap Y_{i_2}$ by Lemma~\ref{lem4.3}.
Without loss of generality we may assume $i_1<i_2$.
Then we have that $\mathbf{M}(\mathrm{F}_{i_2i_1})\circ \mathbf{M}(\mathrm{F}_{i_2i_1})$ contains
a direct summand isomorphic to $\mathrm{G}_{xs}\circ \mathrm{G}_{sy}$, for some
$x,y\in\{1,2,\dots,r\}$. Note that $\mathrm{G}_{xs}\circ \mathrm{G}_{sy}\neq 0$
as $\epsilon_sB\epsilon_s\neq 0$. Consequently, we obtain
\begin{displaymath}
\mathbf{M}(\mathrm{F}_{i_2i_1}\circ \mathrm{F}_{i_2i_1})\cong
\mathbf{M}(\mathrm{F}_{i_2i_1})\circ \mathbf{M}(\mathrm{F}_{i_2i_1})\neq 0.
\end{displaymath}
At the same time, we have $\mathrm{F}_{i_2i_1}\circ \mathrm{F}_{i_2i_1}=0$ since
$\varepsilon_{i_1}A\varepsilon_{i_2}=0$ by our assumption that $A$ is directed,
a contradiction. The claim follows.
\end{proof}

\begin{lemma}\label{lem4.6}
For $i=1,2,\dots,k$, we have $|X_i|=1$.
\end{lemma}

\begin{proof}
Let $X_i=\{s_1,s_2,\dots,s_m\}$. For a fixed $s_q$, the additive closure of
all $\mathrm{F}_{xi}L_{s_q}$, where $x\in\{1,2,\dots,n\}$, is a non-zero
$\cC_A$-invariant subcategory of $B$-proj and hence must coincide with
$B$-proj by transitivity of $\mathbf{M}$. This implies that all
$\mathrm{G}_{s_ys_q}$, where $y\in\{1,2,\dots,m\}$, do appear as direct summands
of $\mathbf{M}(\mathrm{F}_{ii})$.

As $A$ is directed, we have $\varepsilon_i A\varepsilon_i=\Bbbk \varepsilon_i$
and hence $\mathrm{F}_{ii}\circ \mathrm{F}_{ii}\cong \mathrm{F}_{ii}$. Therefore
\begin{equation}\label{eq4.3}
\mathbf{M}(\mathrm{F}_{ii})\circ \mathbf{M}(\mathrm{F}_{ii})\cong \mathbf{M}(\mathrm{F}_{ii})
\end{equation}
as well. Let $h$ denote the multiplicity of $\mathrm{G}_{s_1s_1}$ in $\mathbf{M}(\mathrm{F}_{ii})$.
From the previous paragraph we know that $h>0$. Clearly, $\mathrm{G}_{s_1s_1}$ appears as
a direct summand of $\mathrm{G}_{s_1s_1}\circ \mathrm{G}_{s_1s_1}$. Therefore \eqref{eq4.3}
implies $h^2\leq h$, that is $h=1$.

If $m>1$, then $\mathrm{G}_{s_1s_1}$ appears as
a direct summand of $\mathrm{G}_{s_1s_2}\circ \mathrm{G}_{s_2s_1}$. Therefore in this case
\eqref{eq4.3} fails and we are done.
\end{proof}

As a consequence of Lemmata~\ref{lem4.2} and \ref{lem4.6}, we have $k=r$.
Without loss of generality we may now choose $\epsilon_i$'s such that each
$\mathrm{F}_{ij}$ acts via $\mathrm{G}_{ij}$.

\begin{corollary}\label{cor4.7}
For $i,j=1,2,\dots,k$, we have $\dim(\varepsilon_i A\varepsilon_j)=\dim(\epsilon_i B\epsilon_j)$.
\end{corollary}

\begin{proof}
As each $\mathrm{F}_{st}$ acts via $\mathrm{G}_{st}$, the claim follows by comparing
\begin{displaymath}
\mathrm{F}_{si}\circ \mathrm{F}_{jt}\cong \mathrm{F}_{st}^{\oplus \dim(\varepsilon_i A\varepsilon_j)}
\quad\text{ and }\quad
\mathrm{G}_{si}\circ \mathrm{G}_{jt}\cong \mathrm{G}_{st}^{\oplus \dim(\epsilon_i B\epsilon_j)}.
\end{displaymath}
\end{proof}

\subsection{Completing the proof of Theorem~\ref{thmmain3}}\label{s4.5}

After the preparation in Subsection~\ref{s4.4}, the proof of Theorem~\ref{thmmain3} is similar to
the arguments in \cite[Section~6]{MZ} or \cite[Subsection~4.9]{MaMa}. Consider the principal
$2$-representation $\mathbf{P}_{\mathtt{i}}:=\cC_A(\mathtt{i},{}_-)$ of $\cC_A$, that is the
left regular action of $\cC_A$ on $\cC_A(\mathtt{i},\mathtt{i})$. The additive closure of all
$\mathrm{F}_{i1}$, where $i=1,2,\dots,k$, in $\cC_A(\mathtt{i},\mathtt{i})$ is $\cC_A$-invariant
and gives a $2$-representation which we denote by $\mathbf{N}$. The latter has a unique
$\cC_A$-invariant left ideal $\mathbf{I}$ and the corresponding quotient is exactly the
cell $2$-representation, see \cite[Subsection~6.5]{MM2}.

Mapping $\mathbbm{1}_{\mathtt{i}}$ to $L_1\in B$-mod, gives rise to a $2$-natural transformation $\Phi$ from
$\mathbf{N}$ to $\mathbf{M}$ which, because of the results in Subsection~\ref{s4.4}, sends
indecomposable objects to indecomposable objects. Since, by Corollary~\ref{cor4.7}, the Cartan
matrices of $A$ and $B$ coincide, it follows that $\Phi$ must annihilate $\mathbf{I}$ and
hence induce an equivalence between the cell $2$-representation $\mathbf{N}/\mathbf{I}$
and $\mathbf{M}$. This completes the proof.

\vspace{5mm}

\noindent
V.~M.: Department of Mathematics, Uppsala University, Box. 480,
SE-75106, Uppsala, SWEDEN, email: {\tt mazor\symbol{64}math.uu.se}

\noindent
X.~Z.: Department of Mathematics, Uppsala University, Box. 480,
SE-75106, Uppsala, SWEDEN, email: {\tt xiaoting.zhang\symbol{64}math.uu.se}

\end{document}